\documentclass[review]{siamart171218}

\usepackage{amssymb}
\usepackage{amsmath} 
\usepackage{bbm,alltt}
\usepackage{booktabs} 
\usepackage{tabularx}
\usepackage{float} 
\usepackage{graphicx,epstopdf} 
\ifpdf%
\DeclareGraphicsExtensions{.eps,.pdf,.png,.jpg}
\else
\DeclareGraphicsExtensions{.eps}
\fi
\usepackage[caption=false]{subfig} 
\usepackage{blkarray}
\usepackage{cleveref}	
\usepackage{cite}
\usepackage{mathtools}
\usepackage{algorithm}
\usepackage{algpseudocode}

\usepackage{verbatim} 
\usepackage{lipsum}
\usepackage{amsfonts}
\usepackage{hyperref,todonotes}

\usepackage[scientific-notation=true]{siunitx}

\makeatletter
\newcommand{\leqnomode}{\tagsleft@true\let\veqno\@@leqno}
\newcommand{\reqnomode}{\tagsleft@false\let\veqno\@@eqno}
\makeatother

\numberwithin{theorem}{section}
\numberwithin{algorithm}{section}

\DeclareMathOperator*{\argmin}{argmin} 

\title{Butterfly factorization via randomized matrix-vector multiplications%
\thanks{Version of \today.}}

\author{Yang Liu$^\mathsection$\thanks{Computational Research Division, Lawrence Berkeley National Laboratory, Berkeley, CA 
		(\email{liuyangzhuan@lbl.gov}, \email{pghysels@lbl.gov}, \email{xsli@lbl.gov},).}
	\and Xin Xing\thanks{School of Mathematics, Georgia Institute of Technology, GA 
		(\email{xxing33@gatech.edu}).}
	\and Han Guo\thanks{Department of Electrical Engineering and Computer Science, University of Michigan, MI 
		(\email{hanguo@umich.edu}, \email{emichiel@umich.edu}).}
	\and Eric Michielssen\footnotemark[4]
	\and Pieter Ghysels\footnotemark[2]
	\and Xiaoye Sherry Li\footnotemark[2]
}

 \nolinenumbers

\begin{document}
\maketitle
\begin{abstract}
	This paper presents an adaptive randomized algorithm for computing the butterfly factorization of a $m\times n$ matrix with $m\approx n$ provided that both the matrix and its transpose can be rapidly applied to arbitrary vectors. The resulting factorization is composed of $O(\log n)$ sparse factors, each containing $O(n)$ nonzero entries. The factorization can be attained using $O(n^{3/2}\log n)$ computation and $O(n\log n)$ memory resources. The proposed algorithm applies to matrices with strong and weak admissibility conditions arising from surface integral equation solvers with a rigorous error bound, and is implemented in parallel.  
\end{abstract}



\section{Introduction}
Butterfly factorizations are an attractive means for compressing highly oscillatory operators arising in many scientific and engineering applications, including the integral equation-based analysis of high-frequency acoustic and electromagnetic scattering problems \cite{michielssen_multilevel_1996}, and the evaluation of Fourier integrals and transforms \cite{Candes_butterfly_2009,Lexing_SFT_2009}, spherical harmonic transforms \cite{Tygert_2010_spherical} and many other types of special function transforms \cite{Oneil_2010_specialfunction}. The butterfly factorization of a $m\times n$ matrix with $m\approx n$ exists provided that all judiciously selected submatrices, whose row and column dimensions multiply to $O(n)$, are numerically low-rank. Through recursive low-rank factorizations of these submatrices, the operator can be represented as the product of $O(\log n)$ sparse matrices, each containing $O(n)$ nonzero entries. The resulting factorization can be rapidly applied to arbitrary vectors using only $O(n\log n)$ computation and memory resources. 

Despite this favorable application cost, the cost of constructing a butterfly representation of a given operator typically scales at least as $O(n^2)$ \cite{Tygert_2010_spherical}. Fortunately, there exist two important categories of operators that allow for fast approximation by a butterfly. (i) Operators that allow each element of their matrix representation to be evaluated in O(1) operations. This is typically the case when the butterfly factorization applies directly to an oscillatory operator with an explicit formula (e.g., Fourier operators, special transforms, or discretized integral equations) or stored as a full matrix. (ii) Operators with matrix representations that can be applied to arbitrary vectors in quasi-linear, typically $O(n\log n)$, complexity. This situation typically arises when re-compressing the composition of highly-oscilatory operators, e.g., composition of Fourier integrals \cite{Hormander_1971_FIO,li_butterfly_2015}, matrix algebras for constructing discretized inverse integral operators \cite{Han_2013_butterflyLU}), or conversion to a butterfly representation from other compression formats (e.g, fast multipole methods (FMM)-like formats). 

The butterfly factorization of matrices in category (i) can be constructed using $O(n\log n)$ computation and memory resources following the low-rank decomposition of judiciously selected submatrices using uniform \cite{michielssen_multilevel_1996}, random \cite{Eric_1994_butterfly}, or Chebyshev \cite{Yingzhou_2017_IBF} proxy points. A wide variety of low-rank decompositions, including the interpolative decomposition (ID) \cite{Yingzhou_2017_IBF}, the pseudo skeleton approximation \cite{Eric_1994_butterfly}, the adaptive cross approximation \cite{Tamayo_2011_MLACA}, and singular value decomposition (SVD) \cite{li_butterfly_2015} can be used for this purpose. This type of butterfly factorization has been extended to multi-scale and multi-dimensional problems \cite{michielssen_multilevel_1996,Yingzhou_2015_MBA,Yingzhou_2017_MBF}. 

Operators in category (ii), on the other hand, post bigger challenges to fast butterfly construction algorithms. These algorithms rely on random projection-based algorithms \cite{halko_finding_2011, Liberty_2007_LR} to construct low-rank decompositions of the associated submatrices, typically resulting in higher computation and memory costs than those in category (i). That said, $O(n\log n)$ algorithms can be developed when the oscillatory operator allows for smooth phase recovery \cite{Haizhao_2018_Phase} or fast submatrix-vector multiplications (e.g., using FMM-type algorithms). Unfortunately, these requirements are not met when compressing the concatenation of several Fourier operators or when inverting integral equation operators. When the above conditions do not apply, the cost of reconstructing butterfly approximations for operators in category (ii) require $O(n^{3/2}\log n)$ and $O(n^{3/2})$ computation and memory resources via nested SVD compression \cite{li_butterfly_2015}.              
  
Recently, a $O(n^{3/2}\log n)$-computation but $O(n\log n)$-memory algorithm in Category (ii) for general butterfly compressible matrices was introduced in \cite{Han_2017_butterflyLUPEC}. Unlike the SVD-based algorithm \cite{li_butterfly_2015}  that constructs the factorization from innermost to outermost factors, this algorithm reverses the computation sequence. Consequently, the SVD-based algorithm stores information associated with all random vectors requiring $O(n^{3/2})$ storage, while the algorithm in \cite{Han_2017_butterflyLUPEC} only stores information related to subsets of the structured random vectors, thereby requiring only $O(n\log n)$ storage. 

This paper presents a new butterfly reconstruction scheme that provides a three-fold improvement on the algorithm in \cite{Han_2017_butterflyLUPEC}: (i) The new algorithm is adaptive in nature and permits fast and accurate butterfly reconstructions even when applied to matrices with weak admissibility conditions arising from the discretization of 3D surface integral equation solvers, whereas the previous algorithm results in higher computational complexity. (ii) The algorithm comes with a rigorous error bound obtained using an orthogonal projection argument that grows only weakly with matrix size. (iii) The algorithm can be deployed in parallel. Not surprisingly, the proposed algorithm represents a critical building block for constructing fast iterative and direct solvers for highly-oscillatory problems.   


\section{Preliminary background}
\subsection{Notation}
We use MATLAB notation to denote entries and subblocks of matrices and vectors.
For example, $A(i, j)$ denotes the $(i,j)$th entry of matrix $A$, and $A(I, J)$ with index sets $I$ and $J$ denotes the subblock of matrix $A$ with rows and columns with indices in $I$ and $J$, respectively. Moreover, $\mathrm{diag}(A_1,\ldots,A_k)$ denotes a block diagonal matrix with $k$ diagonal blocks $A_1,\ldots,A_k$. We assume $A \in \mathbb{R}^{m\times n}$ but all the algorithms can be trivially extended to complex matrices. 

\subsection{Low-rank approximation by projection}
Given a matrix $A\in \mathbb{R}^{m\times n}$, we consider a rank-$r$ approximation of $A$ with accuracy $O(\epsilon)$ in the \textit{projection form},
\[
A = U (U^T A) + O(\epsilon),
\]
where $U$ is of dimension $m\times r$ and the symbol ``$T$'' denotes the transpose of a matrix.
The operator $U U^T$ projects all the columns of $A$ onto the $r$-dimensional subspace $\text{col}(U)$. 
The matrix $U$ can be computed via singular value or pivoted QR decompositions, or randomized methods \cite{Liberty_2007_LR}. 
We focus on using randomized methods as illustrated in \Cref{alg:matvec_random} to construct $U$.

%
%

\begin{algorithm}
	\caption{Randomized low-rank approximation method}
	\label{alg:matvec_random}
	\hspace*{\algorithmicindent} \textbf{Input:} Routine to multiply $A \in \mathbb{R}^{m\times n}$ with arbitrary matrices, rank $r$, over-sampling parameter $p$, truncation tolerance $\epsilon$. \\
	\hspace*{\algorithmicindent} \textbf{Output:} Basis matrix $U$ such that $A \approx UU^T A$ with rank at most $(r+p)$. 	
	\begin{algorithmic}[1]
		\State \textbf{Step 1:} Form a $n \times (p+r)$ random matrix $\Omega$ whose entries are random variables which are independent and identically distributed, following a normal distribution.
		
		\State \textbf{Step 2:} Compute $W = A \Omega$. 
		
		\State \textbf{Step 3:} Compute the column-pivoted QR decomposition of $W$ with truncation tolerance $\epsilon$ as $WP = QR$ where $P$ denotes the permutation, $Q$ is orthonormal, and $R$ is upper triangular. Return $U = Q$. 
	\end{algorithmic}
\end{algorithm}

This randomized method has been studied extensively in \cite{halko_finding_2011}. It can be shown that with high probability the rank-$(r+p)$ approximation constructed by \cref{alg:matvec_random} has approximation error similar to that of the optimal rank-$r$ approximation (see Theorem 10.7 of \cite{halko_finding_2011}). \cref{alg:matvec_random} has $O(mnr)$ computational cost when $A$ is stored as a full matrix. Moreover, adaptive versions of \cref{alg:matvec_random} that increase the rank estimate $r$ until the resulting low-rank approximation meets a prescribed approximation accuracy, have been developed \cite{Chris2019Adaptive}.

\section{Butterfly factorization}
We consider the butterfly factorization of a matrix $A= K(T,S)$ defined by a highly-oscillatory operator $K(\cdot,\cdot)$ and point sets $S$ and $T$. Consider the example of free-space wave interactions between 3D \textit{source} points $S$ and \textit{target} points $T$ where $S$ and $T$ are non-overlapping (weak admissibility) or well separated (strong admissibility). The entries $A_{i,j}=K(t_i, s_j) = e^{i2\pi k |t_i-s_j|} / |t_i-s_j|$ for all pairs $(t_i, s_j) \in T\times S$ represent a discretization of the 3D Helmholtz kernel with wavenumber $\kappa>0$. Let $|T| = m$ and $|S| = n$, and assume that $m=O(n)$. Other examples of highly-oscillatory operators includes Green's function operators for Helmholtz equations with nonconstant coefficients, Fourier transforms, and special function transforms. 

\subsection{Hierarchical partitioning}
The point sets $S$ and $T$ are recursively partitioned into small subsets by bisection until the finest subsets thus obtained have less than a prescribed number of points $n_0$. Many geometry or graph-based partitioning algorithms for this purpose have been studied in \cite{rebrova_2018_study}. Here, we use bisection and binary trees for simplicity, and denote the two binary trees $\mathcal{T}_S$ and $\mathcal{T}_T$. We further assume both trees $\mathcal{T}_S$ and $\mathcal{T}_T$ are complete (all levels are completely filled). 
These assumptions can be easily relaxed. 

We number the levels of $\mathcal{T}_T$ and $\mathcal{T}_S$ from the root to the leafs. 
The root node is at level $0$; its children are at level $1$, etc. 
All the leaf nodes are at level $L$. 
At each level $l$, $\mathcal{T}_T$ and $\mathcal{T}_S$ both have $2^l$ nodes. 

Let $T_\tau$ be the subset of points in $T$ corresponding to node $\tau$ in $\mathcal{T}_T$. Obviously, $T_{t} = T$ if $t$ denotes the root node. Furthermore, for any non-leaf node $\tau\in\mathcal{T}_T$ with children $\tau_1$ and $\tau_2$, $T_{\tau_1} \cup T_{\tau_2} = T_\tau$ and $T_{\tau_1} \cap T_{\tau_2} = \emptyset$. 
The same properties hold true for the partitioning of $S$, i.e., $\{S_\nu\}_{\nu\in\mathcal{T}_S}$. We assume that both $\mathcal{T}_T$ and $\mathcal{T}_S$ have the same depth $L = O(\log n)$ so that the leaf point subsets have $O(1)$ points.

\begin{figure}[htbp!]
	\begin{center}
		\begin{tabular}{c}
			\includegraphics[height=1.5in]{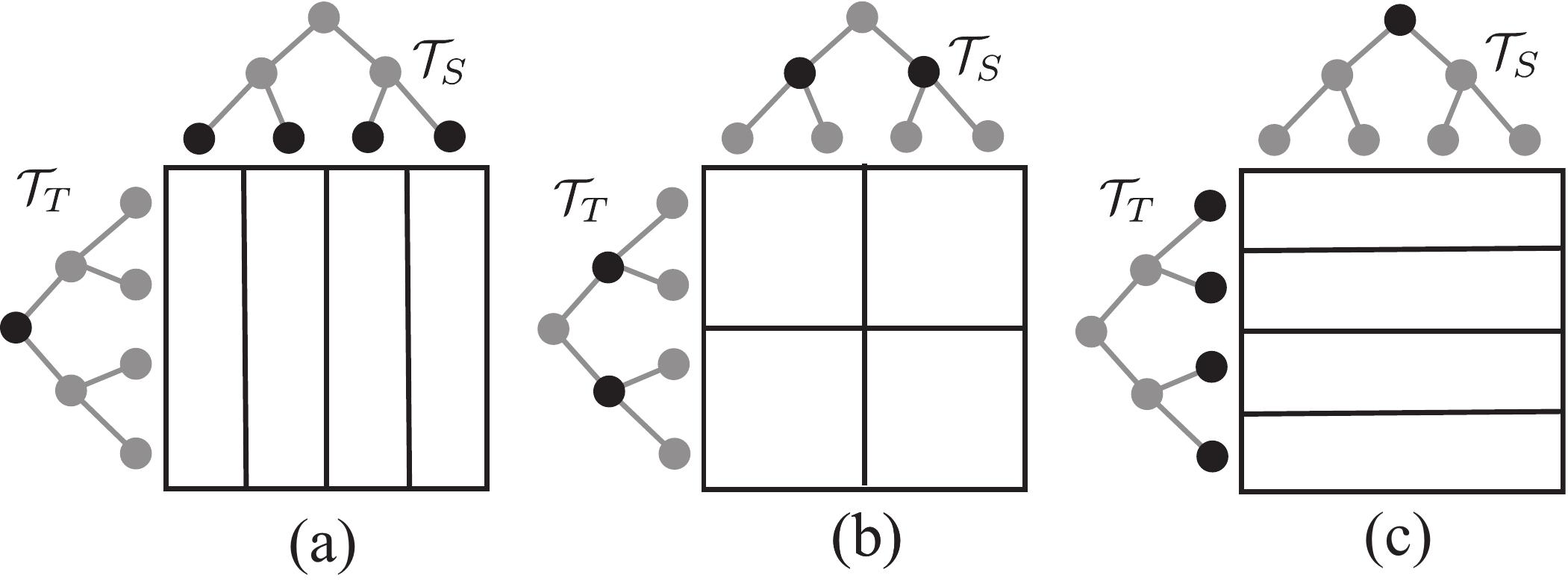}
		\end{tabular}
	\end{center}
	\caption{The partitioning trees $\mathcal{T}_S$, $\mathcal{T}_T$ and the blocks $K(T_\tau, S_\nu)$ for a 2-level butterfly factorization. The blocks correspond to level $l$ of $\mathcal{T}_T$ and level $L-l$ of $\mathcal{T}_S$ with (a) $l=0$ (b) $l=1$ (c) $l=2$.} 
	\label{fig:comp_rank}
\end{figure}

\subsection{Complementary low-rank property}
$A = K(T, S)$ satisfies the \textit{complementary low-rank property} if for any level $0\leq l\leq L$, node $\tau$ at level $l$ of $\mathcal{T}_T$ and a node $\nu$ at level $(L-l)$ of $\mathcal{T}_S$, the subblock $K(T_\tau, S_\nu)$ is numerically low-rank with rank $r_{\tau,\nu}$ bounded by a small number $r$; $r$ is called the (maximum) butterfly rank. See \cref{fig:comp_rank} for example subblocks of a 2-level butterfly. The complementary low-rank property results from the analytic properties of the Helmholtz kernel \cite{Eric_1994_butterfly,Yang_2020_BFprecondition}, and Fourier integrals \cite{Candes_butterfly_2009}, and has been observed for many other highly-oscillatory operators.

\subsection{Low-rank approximation of the blocks}
A butterfly factorization of $A$ compresses all blocks $K(T_\tau, S_\nu)$ with $l = 0, 1, \ldots, L$ for nodes $\tau$ at level $l$ of $\mathcal{T}_T$ and $\nu$ at level $(L-l)$ of $\mathcal{T}_S$ into low-rank form via a nested approach.
There exist three different but equivalent approaches: (1) column-wise butterfly factorization, (2) row-wise butterfly factorization, and (3) hybrid butterfly factorization. The proposed algorithm uses the hybrid factorization as it yields the lowest algorithmic complexity and highest parallel efficiency. That said, we will describe all three approaches as the column- and row-wise factorizations also serve as the building blocks in the hybrid factorization.  

\subsubsection{Column-wise butterfly factorization}\label{sec:column}
Each $K(T_\tau, S_\nu)$ is compressed into rank-$r$ form 
\begin{equation}\label{eqn:lowrank_block}
K(T_\tau, S_\nu) \approx U_{\tau,\nu} U_{\tau,\nu}^T K(T_\tau, S_\nu) = U_{\tau,\nu} E_{\tau,\nu},
\end{equation}
where $U_{\tau,\nu}\in \mathbb{R}^{|T_\tau| \times r}$ has orthonormal columns and is referred to as \textit{the column basis matrix} associated with $(\tau,\nu)$. The compression is performed directly for any $T_\tau$ at the leaf level $L$ of $\mathcal{T}_T$.

At a non-leaf level $l < L$, consider a node $\tau$ at level $l$ of $\mathcal{T}_T$ and a node $\nu$ at level $(L-l)$ of $\mathcal{T}_S$. 
Let $\{\tau_1, \tau_2\}$ be the children of $\tau$ at level $(l+1)$ of $\mathcal{T}_T$ and $p_\nu$ be the parent node of $\nu$ at level $(L-l-1)$ of $\mathcal{T}_S$. 
The low-rank approximation \cref{eqn:lowrank_block} of $K(T_\tau, S_\nu)$ is constructed using the low-rank approximations of blocks $K(T_{\tau_1}, S_{p_\nu})$ and $K(T_{\tau_2}, S_{p_\nu})$ at level $(l+1)$ using the nested compression approach described next.

Since $T_\tau = T_{\tau_1} \cup T_{\tau_2}$, $K(T_\tau, S_\nu)$ can be split into two blocks, i.e., 
\begin{equation}\label{eqn:split}
K(T_\tau, S_\nu) 
= 
\begin{bmatrix}
K(T_{\tau_1}, S_\nu) \\
K(T_{\tau_2}, S_\nu)
\end{bmatrix}.
\end{equation}
Meanwhile, since $S_\nu$ is a subset of $S_{p_\nu}$, it follows that $K(T_{\tau_a}, S_\nu)$, for each child $\tau_a$ of $\tau$, is a subblock of $K(T_{\tau_a}, S_{p_\nu})$. 
Thus, the low-rank approximation $K(T_{\tau_a}, S_{p_\nu}) \approx U_{\tau_a, p_\nu} E_{\tau_a, p_\nu}$ in \cref{eqn:lowrank_block} yields a low-rank approximation of $K(T_{\tau_a}, S_\nu)$ as,
\[
K(T_{\tau_a}, S_\nu) \approx U_{\tau_a, p_\nu} E_{\tau_a, \nu},
\]
where $E_{\tau_a, \nu}$ contains a subset of columns in $E_{\tau_a, p_\nu}$ corresponding to $S_\nu$. 
Substituting this approximation into \cref{eqn:split} yields 
\begin{equation}\label{eqn:nested_approx1}
K(T_\tau, S_\nu) 
\approx 
\begin{bmatrix}
U_{\tau_1, p_\nu} E_{\tau_1, \nu} \\
U_{\tau_2, p_\nu} E_{\tau_2, \nu}
\end{bmatrix}
= 
\begin{bmatrix}
U_{\tau_1, p_\nu} &  \\
& U_{\tau_2, p_\nu}
\end{bmatrix}
\begin{bmatrix}
E_{\tau_1, \nu} \\
E_{\tau_2, \nu}
\end{bmatrix}. 
\end{equation}

Instead of directly compressing $K(T_\tau, S_\nu)$, we compute a rank-$r$ approximation of the last matrix above, with $2r$ rows, far fewer than the original matrix $K(T_\tau, S_\nu)$, as 
\begin{equation}\label{eqn:nested_approx}
\begin{bmatrix}
E_{\tau_1, \nu} \\
E_{\tau_2, \nu}
\end{bmatrix} 
\approx 
R_{\tau,\nu} E_{\tau,\nu},
\end{equation}
where $R_{\tau,\nu}$ has orthonormal columns. 
Substituting \cref{eqn:nested_approx} into \cref{eqn:nested_approx1}, we obtain a rank-$r$ approximation of $K(T_\tau, S_\nu)$ as,
\[
K(T_\tau, S_\nu) 
\approx 
\begin{bmatrix}
U_{\tau_1, p_\nu} & \\
& U_{\tau_2, p_\nu}
\end{bmatrix}
R_{\tau,\nu} E_{\tau,\nu} 
= U_{\tau,\nu} E_{\tau,\nu},
\]
where the column basis matrix $U_{\tau,\nu}$ is
\begin{equation}\label{eqn:nested_basis}
U_{\tau,\nu} =
\begin{bmatrix}
U_{\tau_1, p_\nu} & \\
& U_{\tau_2, p_\nu}
\end{bmatrix}
R_{\tau,\nu},
\end{equation}
and $R_{\tau,\nu}$ is referred to as a \textit{transfer matrix}; note that $U_{\tau,\nu}$ still has orthonormal columns. 

Using \cref{eqn:nested_basis}, the basis matrices at any non-leaf level $l$ are expressed in terms of the basis matrices at level $(l+1)$ via the transfer matrices. 
Thus, the basis matrices at any non-leaf level are not explicitly formed but instead recovered recursively from quantities at lower levels. 
In the end, the butterfly factorization of $K(T, S)$ consists of the low-rank approximations of blocks at level $0$ of $\mathcal{T}_T$, i.e., 
\begin{align}
K(T, S) 
&= 
\begin{bmatrix}
K(T_t, S_{\nu_1}) & K(T_t, S_{\nu_2}) & \ldots & K(T_t, S_{\nu_{2^L}})
\end{bmatrix} 
\nonumber \\ & \approx 
\begin{bmatrix}
U_{t,\nu_1}E_{t, \nu_1} & U_{t,\nu_2}E_{t, \nu_2} & \ldots & U_{t,\nu_{2^L}} E_{t, \nu_{2^L}}
\end{bmatrix} \label{eqn:column_wise_butterfly} \\ 
& =\big(U^LR^{L-1}R^{L-2}\ldots R^{0}\big)E^0
\label{eqn:column_wise_butterfly_mat}
\end{align}
where $t$ denotes the root node of $\mathcal{T}_T$, $\{\nu_1, \nu_2, \ldots, \nu_{2^L}\}$ contains all the leafs of $\mathcal{T}_S$, and $U_{t,\nu_1}, U_{t,\nu_2}, \ldots$ are the corresponding column basis matrices. Expanding each $U_{t, \nu_a}$ using the nested form \cref{eqn:nested_basis} up to the leaf level of $\mathcal{T}_T$, $K(T, S)$ can be represented as the product of the $(L+2)$ matrices in \cref{eqn:column_wise_butterfly_mat} where $U^L=\mathrm{diag}(U_{\tau_1,s},\ldots,U_{\tau_{2^L},s})$ consists of column basis matrices at level $L$, $E^0=\mathrm{diag}(E_{t,\nu_1},\ldots,E_{t,\nu_{2^L}})$, and each factor $R^l, l=0,\ldots,L-1$ is block diagonal consisting of diagonal blocks $R_{\nu}$ for all nodes $\nu$ at level $l$ of $\mathcal{T}_S$. Each $R_\nu$ consists of $R_{\tau,\nu_1}$ and $R_{\tau,\nu_2}$ for all nodes $\tau$ at level $L-l$ of $\mathcal{T}_T$, as  
\begin{align}
R_\nu=
\begin{bmatrix}
R_{\tau_1,\nu_1} &&&&\!\!\!\!\!\!\!\!R_{\tau_1,\nu_2} &&&\\
 &\!\!\!\!\!\!\!\!R_{\tau_2,\nu_1}&&&& \!\!\!\!\!\!\!\!R_{\tau_2,\nu_2} &  & &\\
 &  & \ddots &   &  & & \ddots & \\
 &  &  & \!\!\!\!\!\!\!\!R_{\tau_{2^{L-l}},\nu_1}  &  & &  & \!\!\!\!\!\!\!\!R_{\tau_{2^{L-l}},\nu_2}
\end{bmatrix} 
\end{align}
where $\{\nu_1,\nu_2\}$ denotes the children of $\nu$. We term $U^L$ and $E^0$ outer factors and $R^l$ inner factors. Figure \ref{fig:structure}(a) shows an example of a 4-level column-wise factorization.   

\begin{figure}[htbp!]
	\begin{center}
		\begin{tabular}{c}
			\includegraphics[height=2.8in]{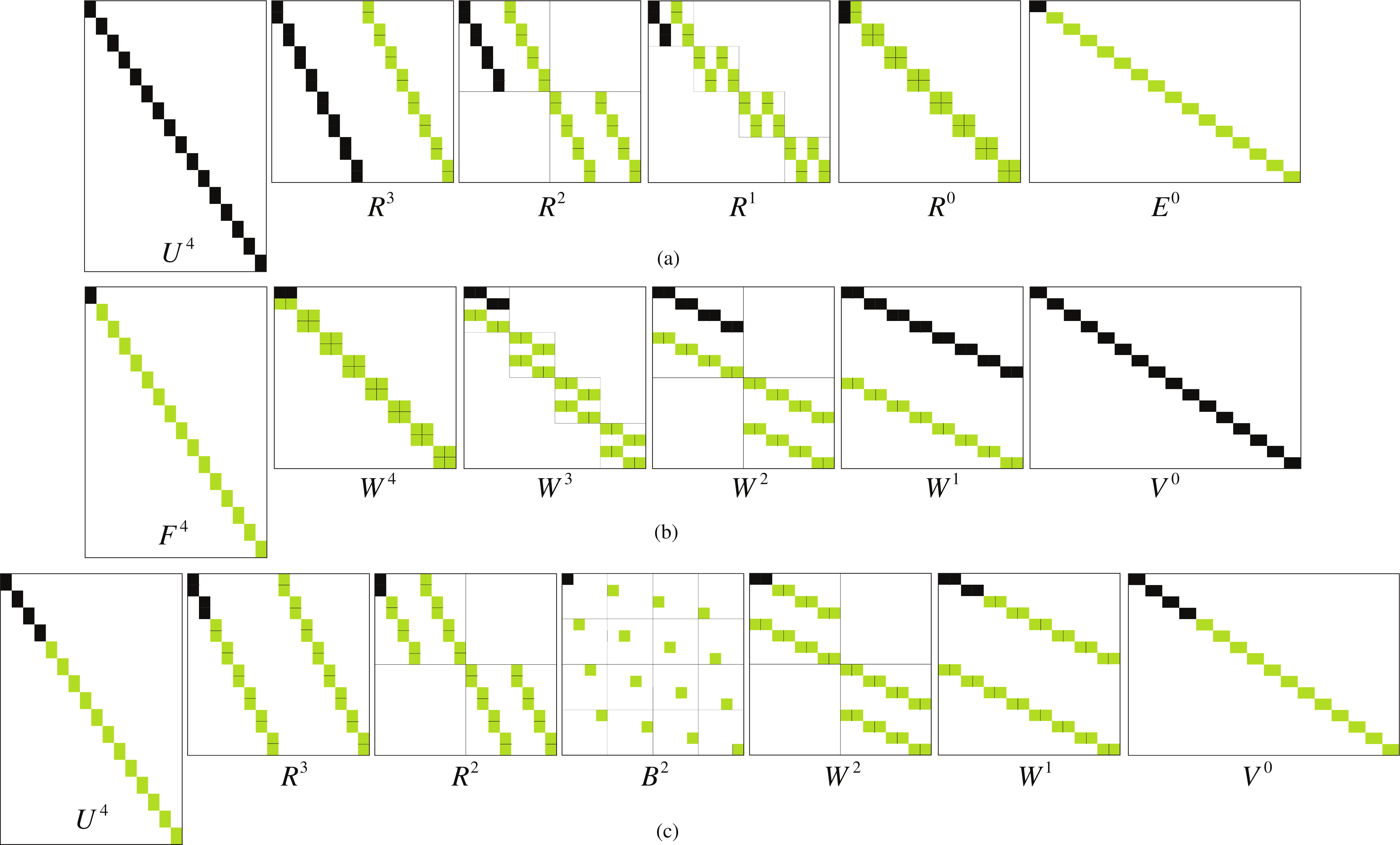}
		\end{tabular}
	\end{center}
	\caption{(a) Column-wise, (b) row-wise, and (c) hybrid factorizations for a 4-level butterfly. The blocks in black multiply to $U_{t,\nu_1}E_{t, \nu_1}$ in \cref{eqn:column_wise_butterfly}, $F_{\tau_1, s} V_{\tau_1, s}^T$, and $U_{\tau_1,\nu_1} B_{\tau_1,\nu_1} V_{\tau_1,\nu_1}^T$ in \cref{eqn:hybrid_butterfly}, respectively.}
	\label{fig:structure}
\end{figure}

\subsubsection{Row-wise butterfly factorization}\label{sec:row}
Each $K(T_\tau, S_\nu)$ is compressed into rank-$r$ form as 
\begin{equation}\label{eqn:lowrank_block_col}
K(T_\tau, S_\nu) \approx K(T_\tau, S_\nu) V_{\tau,\nu} V_{\tau,\nu}^T = F_{\tau,\nu} V_{\tau,\nu}^T,
\end{equation}
where $V_{\tau,\nu}\in \mathbb{R}^{|S_\nu| \times r}$ has orthonormal columns and is referred to as \textit{the row basis matrix} associated with $(\tau,\nu)$. Just like for the column-wise factorization, we define the transfer matrix $W_{\tau,\nu}$ for a non-leaf node $\nu$ as 
\begin{equation}\label{eqn:nested_basis_col}
V_{\tau,\nu} =
\begin{bmatrix}
V_{p_\tau, \nu_1} & \\
& V_{p_\tau, \nu_2}
\end{bmatrix}
W_{\tau,\nu}.
\end{equation}

The basis and transfer matrices can be constructed upon applying the column-wise butterfly factorization to $K(T, S)^T$, yielding the row-wise butterfly structure
\begin{align}
K(T, S)=F^L\big(W^LW^{L-1}\ldots W^1V^0\big) \label{eqn:row_wise_butterfly_mat}
\end{align} 

Here the outer factors are $F^L=\mathrm{diag}(F_{\tau_1,s},\ldots,F_{\tau_{2^L},s})$ with $s$ denoting the root of $\mathcal{T}_S$ and $V^0=\mathrm{diag}(V^T_{t,\nu_1},\ldots,V^T_{t,\nu_{2^L}})$, and the block-diagonal inner factors $W^l, l=1,\ldots,L$ have blocks $W_{\tau}$ for all nodes $\tau$ at level $l-1$ of $\mathcal{T}_T$. Each $W_\tau$ consists of $W_{\tau_1,\nu}$ and $W_{\tau_2,\nu}$ for all nodes $\nu$ at level $L-l+1$ of $\mathcal{T}_S$, as  
\begin{align}
W_\tau=
\begin{bmatrix}
W_{\tau_1,\nu_1} &&&&\\
&W_{\tau_1,\nu_2}&&&\\
&  & \ddots &    \\
&  &  & W_{\tau_1,\nu_{2^{L-l+1}}} \\
W_{\tau_2,\nu_1} &&&&\\
&W_{\tau_2,\nu_2}&&&\\
&  & \ddots &    \\
&  &  & W_{\tau_2,\nu_{2^{L-l+1}}} \\
\end{bmatrix} .
\end{align}
Figure \ref{fig:structure}(b) shows an example of a 4-level row-wise factorization.

\subsubsection{Hybrid butterfly factorization}\label{sec:hybrid}

At any level $l$ of $\mathcal{T}_T$, $K(T_\tau, S_\nu)$ with all nodes $\tau$ at level $l$ of $\mathcal{T}_T$ and nodes $\nu$ at level $(L-l)$ of $\mathcal{T}_S$  (referred to as the blocks at level $l$) form a non-overlapping partitioning of $K(T, S)$. 
Fixing $l$, we can combine the computed row and column basis matrices $U_{\tau,\nu}$, $V_{\tau,\nu}$ from both the column-wise and row-wise butterfly factorizations of $K(T, S)$ above to compress $K(T_\tau, S_\nu)$ as 
\[
K(T_\tau, S_\nu) \approx U_{\tau,\nu} U_{\tau,\nu}^T K(T_\tau, S_\nu) V_{\tau,\nu} V_{\tau,\nu}^T = U_{\tau,\nu} B_{\tau,\nu} V_{\tau,\nu}^T.
\]
The hybrid butterfly factorization of $K(T, S)$ is constructed as, 
\begin{align}
K(T, S) 
& =
\begin{bmatrix}
K(T_{\tau_1}, S_{\nu_1}) & K(T_{\tau_1}, S_{\nu_2}) & \cdots & K(T_{\tau_1}, S_{\nu_q}) \\
K(T_{\tau_2}, S_{\nu_1}) & K(T_{\tau_2}, S_{\nu_2}) & \cdots & K(T_{\tau_2}, S_{\nu_q}) \\
\vdots & \vdots & \ddots & \vdots \\
K(T_{\tau_p}, S_{\nu_1}) & K(T_{\tau_p}, S_{\nu_2}) & \cdots & K(T_{\tau_p}, S_{\nu_q})
\end{bmatrix} 
\nonumber \\ & \approx 
\begin{bmatrix}
U_{\tau_1,\nu_1} B_{\tau_1,\nu_1} V_{\tau_1,\nu_1}^T & U_{\tau_1,\nu_2} B_{\tau_1,\nu_2} V_{\tau_1,\nu_2}^T  & \cdots & U_{\tau_1,\nu_q} B_{\tau_1,\nu_q} V_{\tau_1,\nu_q}^T \\
U_{\tau_2,\nu_1} B_{\tau_2,\nu_1} V_{\tau_2,\nu_1}^T & U_{\tau_2,\nu_2} B_{\tau_2,\nu_2} V_{\tau_2,\nu_2}^T  & \cdots & U_{\tau_2,\nu_q} B_{\tau_2,\nu_q} V_{\tau_2,\nu_q}^T \\
\vdots & \vdots & \ddots & \vdots \\
U_{\tau_p,\nu_1} B_{\tau_p,\nu_1} V_{\tau_p,\nu_1}^T & U_{\tau_p,\nu_2} B_{\tau_p,\nu_2} V_{\tau_p,\nu_2}^T  & \cdots & U_{\tau_p,\nu_q} B_{\tau_p,\nu_q} V_{\tau_p,\nu_q}^T 
\end{bmatrix} \label{eqn:hybrid_butterfly}\\
& = \big(U^LR^{L-1}R^{L-2}\ldots R^{l}\big)B^{l}\big(W^lW^{l-1}\ldots W^1V^0\big)
\label{eqn:hybrid_butterfly_mat}
\end{align}
where $\tau_1, \tau_2, \ldots, \tau_p$ are the $p=2^l$ nodes at level $l$ of $\mathcal{T}_T$,  and $\nu_1, \nu_2, \ldots, \nu_q$ are the $q=2^{L-l}$ nodes at level $(L-l)$ of $\mathcal{T}_S$. 

In this level-$l$ hybrid butterfly factorization, these column basis matrices $U_{\tau,\nu}$ are recursively defined as in \cref{eqn:nested_basis} using column basis and transfer matrices in the lower levels of $\mathcal{T}_T$, while the row basis matrices $V_{\tau,\nu}$ are recursively defined as in \cref{eqn:nested_basis_col} using row basis and transfer matrices at upper levels of $\mathcal{T}_T$. In \cref{eqn:hybrid_butterfly}, the outer factors $U^L,V^0$ and inner factors $R^a,W^a$ are defined in Sections \ref{sec:column} and \ref{sec:row}, and the inner factor $B^l$ consists of blocks $B_{\tau,\nu}$ at level $l$ in \cref{eqn:hybrid_butterfly_mat}. Typically, the level $l$ is set to $l_m = \lfloor L/2\rfloor$. Figure \ref{fig:structure}(c) shows a hybrid factorization with $L=4$ and $l_m=2$.

\begin{figure}[htbp!]
	\begin{center}
		\begin{tabular}{c}
			\includegraphics[height=3in]{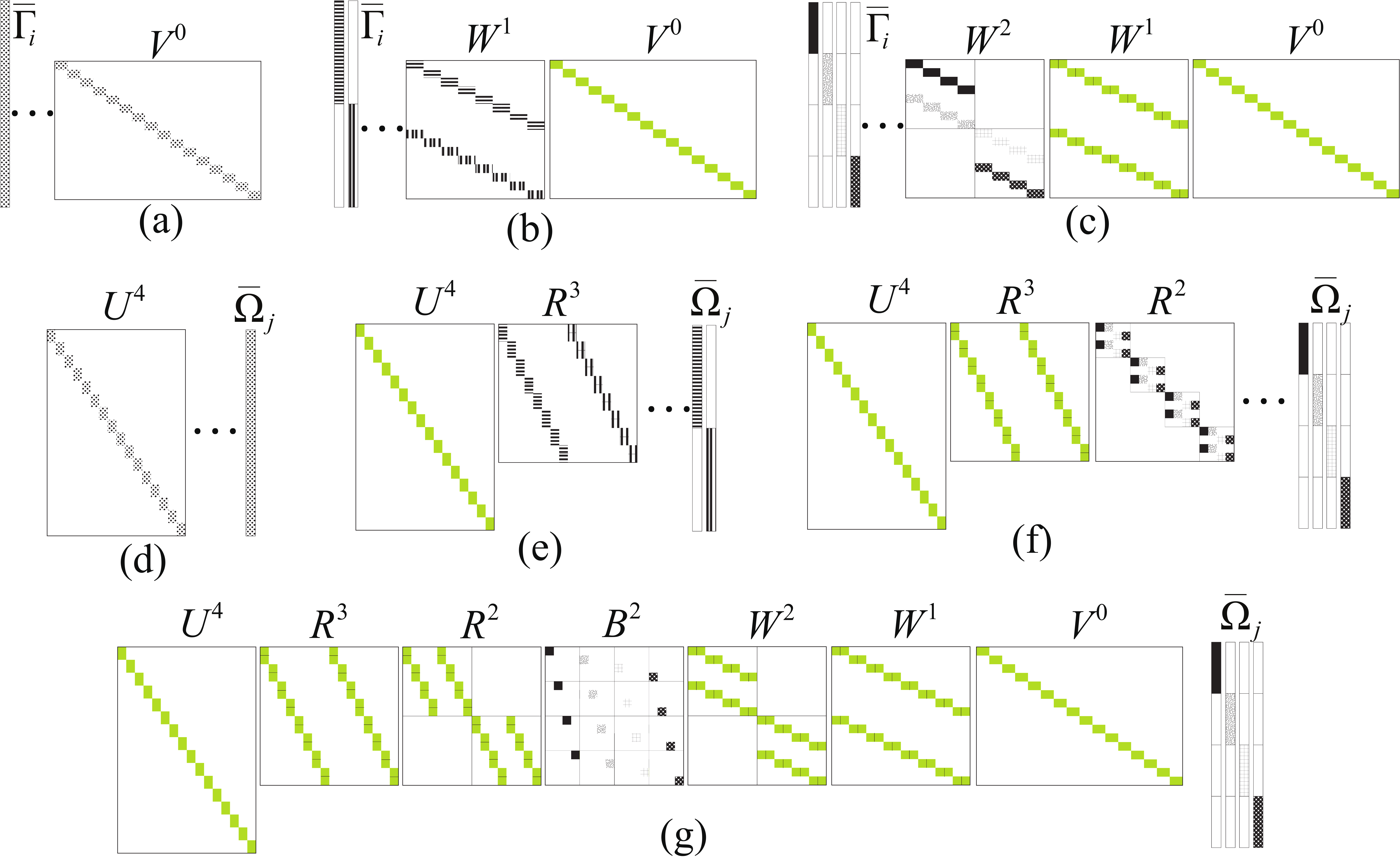}
		\end{tabular}
	\end{center}
	\caption{A 4-level hybrid factorization based on matrix-vector products consists of steps that compute (a) $V^0$, (b) $W^0$, (c) $W^1$, (d) $U^4$, (e) $R^3$, (f) $R^2$, and (g) $B^2$. Note that the vectors $\bar{\Omega}_\nu$ or $\bar{\Gamma}_\tau$ and the blocks being computed are marked with the same texture. The already-computed blocks needed at each step are plotted in Green.} 
	\label{fig:update_order}
\end{figure}

\section{Adaptive butterfly factorization  via randomized matrix-vector products}
We propose an algorithm for constructing the butterfly factorization of a matrix $A=K(T,S)$ using only products of $A$ and its transpose with random vectors. The proposed \cref{alg:matvec_butterfly} returns a hybrid factorization with prescribed accuracy $\epsilon$ assuming black-box matrix-vector multiplications. With minor modifications, \cref{alg:matvec_butterfly} also applies to column- and row-wise factorizations albeit with a much higher computational cost. In what follows, we describe the four key components of the algorithm, including
\begin{itemize}
	\item The computation of $K(T_\tau, S_\nu)\Omega_\nu$ and $K(T_\tau, S_\nu)^T\Gamma_\tau$ with random matrices $\Omega_\nu$ and $\Gamma_\tau$ using a black-box routine.	
	\item The construction of column basis matrices $U_{\tau,\nu}$ (or transfer matrices $R_{\tau,\nu}$ for non-leaf nodes $\tau$ in $\mathcal{T}_T$) based on matrix-vector multiplications involving $K(T, S)$.
	\item The construction of row basis matrices $V_{\tau,\nu}$ (or transfer matrices $W_{\tau,\nu}$ for non-leaf nodes $\nu$ in $\mathcal{T}_S$) based on the matrix-vector multiplications involving $K(T, S)^T$.
	\item The construction of intermediate matrices $B_{\tau_a,\nu_b}$ in \cref{eqn:hybrid_butterfly}.
\end{itemize}

\begin{algorithm}
	\caption{Adaptive and randomized hybrid butterfly factorization based on matrix-vector multiplication}
	\label{alg:matvec_butterfly}
	\hspace*{\algorithmicindent} \textbf{Input:} Black-box routine for multiplying $K(T,S) \in \mathbb{R}^{m\times n}$ and its transpose with arbitrary matrices, over-sampling parameter $p$, truncation tolerance $\epsilon$, initial rank guess $r_0$, binary partitioning trees $\mathcal{T}_S$ and $\mathcal{T}_T$ of $L$ levels. \\
	\hspace*{\algorithmicindent} \textbf{Output:} $K(T,S)\approx (U^LR^{L-1}R^{L-2}\ldots R^{l_m})B^{l_m}(W^{l_m}W^{l_m-1}\ldots W^1V^0)$ with $l_m=\lfloor L/2\rfloor$. 	
	\begin{algorithmic}[1]
		\State $r=r_0$.
		\While{not conerged}\Comment{Adaptive computation of $V_{\tau,\nu}$ at level $0$}\label{line:loopV}
		\State Form a random matrix $\Gamma_t\in \mathbb{R}^{|T_t|\times (r+p)}$ for the root node $t$ of $\mathcal{T}_T$.
		\State Compute $K(T_t,S)^T\Gamma_t$. \label{line:matvecV}
		\For{$\nu$ at level $L$ of $\mathcal{T}_S$}
		\State Apply \cref{alg:matvec_random} with $A=K(T_t,S_\nu)^T$ to compute $V_{t,\nu}$.
		\EndFor
		\State Converge if $r>\max\limits_{\nu} \{r_{t,\nu}\}$. \label{line:Vstop}\Comment{Over all nodes $\nu$ at leaf level of $\mathcal{T}_S$}
		\State $r\leftarrow 2r$.
		\EndWhile
		
		\While{not conerged}\Comment{Adaptive computation of $U_{\tau,\nu}$ at level $L$}\label{line:loopU}
		\State Form a random matrix $\Omega_s\in \mathbb{R}^{|S_s|\times (r+p)}$ for the root node $s$ of $\mathcal{T}_S$.
		\State Compute $K(T,S_s)\Omega_s$. \label{line:matvecU} 
		\For{$\tau$ at level $L$ of $\mathcal{T}_T$}
		\State Apply \cref{alg:matvec_random} with $A=K(T_\tau,S_s)$ to compute $U_{\tau,s}$.
		\EndFor
		\State Converge if $r>\max\limits_{\tau} \{r_{\tau,s}\}$. \label{line:Ustop}\Comment{Over all nodes $\tau$ at leaf level of $\mathcal{T}_T$}
		\State $r\leftarrow 2r$.
		\EndWhile

		\For{$l=1$ to $l_m$}\Comment{Computation of $W_{\tau,\nu}$}
		\For{$\tau$ at level $l$ of $\mathcal{T}_T$}
		\State Form a random matrix $\Gamma_\tau\in \mathbb{R}^{|T_\tau|\times (r+p)}$ with $r=\max\limits_{\nu} \{r_{p_\tau,\nu_1}+r_{p_\tau,\nu_2}\}$. \label{line:rGamma}\Comment{Over all nodes $\nu$ at level $L-l$ of $\mathcal{T}_S$}
		\State Compute $K(T_\tau,S)^T\Gamma_\tau$. \label{line:matvecW} 
		\For{$\nu$ at level $L-l$ of $\mathcal{T}_S$}
		\State Compute $\begin{bmatrix}
		V_{p_\tau, \nu_1}^T & \\
		& V_{p_\tau, \nu_2}^T 
		\end{bmatrix}K(T_\tau,S_\nu)^T\Gamma_\tau$, apply \cref{alg:matvec_random} for $W_{\tau,\nu}$. \label{line:W_compute}
		\EndFor				
		\EndFor
		\EndFor
		
		\For{$l=L-1$ to $l_m$}\Comment{Computation of $R_{\tau,\nu}$ and $B_{\tau,\nu}$}
		\For{$\nu$ at level $L-l$ of $\mathcal{T}_S$}
		\State Form a random matrix $\Omega_\nu\in \mathbb{R}^{|S_\nu|\times (r+p)}$ with $r=\max\limits_{\tau} \{r_{\tau_1,p_\nu}+r_{\tau_2,p_\nu}\}$.\label{line:rOmega}\Comment{Over all nodes $\tau$ at level $l$ of $\mathcal{T}_T$}
		\State Compute $K(T,S_\nu)\Omega_\nu$.\label{line:matvecR} 
		\State Compute $V^T_{\tau,\nu}\Omega_\nu$ if $l=l_m$. \Comment $V^T_{\tau,\nu}$ is not explicitly computed. \label{line:matvecHalf} 
		\For{$\tau$ at level $L-l$ of $\mathcal{T}_T$}
		\State Compute $\begin{bmatrix}
		U_{\tau_1, p_\nu}^T & \\
		& U_{\tau_2, p_\nu}^T 
		\end{bmatrix}K(T_\tau,S_\nu)\Omega_\nu$, apply \cref{alg:matvec_random} for $R_{\tau,\nu}$. \label{line:R_compute}
		\State Compute $B_{\tau,\nu}$ using \cref{eqn:Bij} if $l=l_m$.
		\EndFor				
		\EndFor
		\EndFor
	\end{algorithmic}
\end{algorithm}

\subsection{Multiplication of $K(T_\tau, S_\nu)$ and $K(T_\tau, S_\nu)^T$ by random matrices}\label{sec:matvec}
Here we assume the existence of a black-box program to perform matrix-vector multiplications involving $K(T, S)$ and $K(T, S)^T$. 
To use \cref{alg:matvec_random} to compute the column/row basis matrices for each block $K(T_\tau, S_\nu)$, we multiply $K(T, S)$ or $K(T, S)^T$ with \textit{structured random matrices} $\bar{\Omega}_\nu$ and $\bar{\Gamma}_\tau$ to obtain the matrices $K(T_\tau, S_\nu) \Omega_\nu$ and $K(T_\tau, S_\nu)^T \Gamma_\tau$ that are used in \cref{alg:matvec_random}. As we shall see later, $\Omega_\nu$ and $\Gamma_\tau$ are sub-vectors of $\bar{\Omega}_\nu$ and $\bar{\Gamma}_\tau$. Their entries are random variables which are independent and identically distributed, following a normal distribution. 

Fixing a node $\nu$ at level $(L-l)$ of $\mathcal{T}_S$, we compute $K(T_\tau, S_\nu)\Omega_\nu$ with a $|S_\nu|\times (r+p)$ matrix $\Omega_\nu$ for all nodes $\tau$ at level $l$ of $\mathcal{T}_T$, by multiplying $K(T, S)$ with a sparse $|S|\times (r+p)$ matrix $\bar{\Omega}_\nu$ whose only non-zero entries $\Omega_\nu$ are located on the rows corresponding to $S_\nu$. To evaluate all the multiplications $K(T_\tau, S_\nu)\Omega_\nu$ at level $l$ (of $\mathcal{T}_T$), $2^{(L-l)} (r+p)$ matrix-vector multiplications by $K(T,S)$ are needed. 

Similarly, fixing a node $\tau$ at level $l$ of $\mathcal{T}_T$, 
we compute $K(T_\tau, S_\nu)^T \Gamma_\tau$ with a $|T_\tau|\times (r+p)$ matrix $\Gamma_\tau$ for all nodes $\nu$ at level $(L-l)$ of $\mathcal{T}_S$, by multiplying $K(T, S)^T$ with a sparse matrix $\bar{\Gamma}_\tau \in \mathbb{R}^{|T|\times (r+p)}$ whose only non-zero entries $\Gamma_\tau$ are located on the rows corresponding to $T_\tau$. To evaluate all the multiplications $K(T_\tau, S_\nu)^T \Gamma_\tau$ at level $l$ (of $\mathcal{T}_T$), 
$2^l(r+p)$ matrix-vector multiplications by $K(T, S)^T$ are needed. In \cref{alg:matvec_butterfly}, the products $K(T_\tau, S_\nu) \Omega_\nu$ and $K(T_\tau, S_\nu)^T \Gamma_\tau$ are performed on lines \ref{line:matvecV}, \ref{line:matvecW}, \ref{line:matvecU}, \ref{line:matvecR}. 

\subsection{Computation of $U_{\tau,\nu}$ and $R_{\tau,\nu}$}\label{sec:compute_R}
For each leaf node $\tau$ at level $L$ of $\mathcal{T}_T$ and the root node $s$ of $\mathcal{T}_S$, 
the column basis matrix $U_{\tau, s}$ can be directly computed using \cref{alg:matvec_random} by evaluating $K(T_\tau, S_s) \Omega_s$. In \cref{alg:matvec_butterfly} (line \ref{line:loopV}), the ranks of $U_{\tau,s}$ are determined by adaptively doubling the size of the random matrices $\Omega_s$ for $U_{\tau,s}$ until convergence. The iteration is terminated if the rank estimate $r$ exceeds the maximum revealed rank $\max\limits_{\nu} \{r_{t,\nu}\}$ (see line \ref{line:Ustop}). This heuristic stopping criterion is used as a more rigorous criterion requires expensive computation of the approximation error.   

For each non-leaf node $\tau$ at level $l_m\leq l < L$ of $\mathcal{T}_T$ and each node $\nu$ at level $(L-l)$ of $\mathcal{T}_S$, 
the matrix to be compressed in \cref{eqn:nested_approx1} when computing $R_{\tau,\nu}$ can be expressed as, 
\[
\begin{bmatrix}
E_{\tau_1, \nu}\\
E_{\tau_2, \nu}
\end{bmatrix}
= 
\begin{bmatrix}
U_{\tau_1, p_\nu}^T & \\
& U_{\tau_2, p_\nu}^T 
\end{bmatrix}
K(T_\tau, S_\nu)
\approx R_{\tau,\nu} E_{\tau,\nu}. 
\]
Thus, $R_{\tau,\nu}$ can be computed via \cref{alg:matvec_random} using the products
\[
\begin{bmatrix}
U_{\tau_1, p_\nu}^T & \\
& U_{\tau_2, p_\nu}^T 
\end{bmatrix}
K(T_\tau, S_\nu)\Omega_\nu. 
\]
Note that no rank adaptation is needed for $R_{\tau,\nu}$ in \cref{alg:matvec_butterfly} as the rank $r_{\tau,\nu}$ is bounded by $r_{\tau_1,p_\nu}+r_{\tau_2,p_\nu}$. The dimensions of the random matrices $\Omega_\nu$ therefore are chosen using the rank estimate $r=\max\limits_{\nu} \{r_{\tau_1,p_\nu}+r_{\tau_2,p_\nu}\}$ on line \ref{line:rOmega}. This process recursively traverses $\mathcal{T}_T$ from the leafs to center level $l_m$ and $\mathcal{T}_S$ from the root to center level $l_m$.

\subsection{Computation of $V_{\tau,\nu}$ and $W_{\tau,\nu}$}\label{sec:compute_W}
The computation of $V_{\tau,\nu}$ and $W_{\tau,\nu}$ resembles the above computation of $U_{\tau,\nu}$ and $R_{\tau,\nu}$, but uses the multiplication results $K(T_\tau, S_\nu)^T \Gamma_\tau$. $V_{\tau,\nu}$ for leaf nodes is computed adaptively on line \ref{line:loopU} while $W_{\tau,\nu}$ for non-leaf nodes is computed using \cref{alg:matvec_random} using the multiplication results
\[
\begin{bmatrix}
V_{p_\tau, \nu_1}^T & \\
& V_{p_\tau, \nu_2}^T 
\end{bmatrix}
K(T_\tau, S_\nu)^T\Gamma_\tau. 
\] 
where the dimensions of the random matrices $\Gamma_\tau$ are chosen using the rank estimate on line \ref{line:rGamma} without adaptation. This recursive construction starts from the leaf level of $\mathcal{T}_S$ and ends at center level $l_m$.

\subsection{Computation of $B_{\tau,\nu}$ at level $l_m$}
It follows from \cref{eqn:hybrid_butterfly} that for each node $\tau$ at level $l_m$ of $\mathcal{T}_T$ and node $\nu$ at level $(L-l_m)$ of $\mathcal{T}_S$, $K(T_\tau, S_\nu)$ is approximated as 
\[
K(T_\tau, S_\nu) \approx U_{\tau,\nu} B_{\tau, \nu} V_{\tau, \nu}^T. 
\]
Using the existing multiplication results $K(T_\tau, S_\nu) \Omega_\nu$, 
$B_{\tau, \nu}$ can be estimated as
\begin{align}
B_{\tau, \nu} 
&= 
\argmin_{B\in \mathbb{R}^{r_{\tau,\nu} \times r_{\tau,\nu}}} \| \left(K(T_\tau, S_\nu) \Omega_\nu\right) - U_{\tau,\nu} B V_{\tau,\nu}^T \Omega_\nu\|_F
\nonumber\\
& = 
U_{t,\nu}^T \left(K(T_t, S_\nu) \Omega_\nu\right)  (V_{\tau,\nu}^T \Omega_\nu)^T.\label{eqn:Bij}
\end{align}

Note that \cref{alg:matvec_butterfly} computes and stores $K(T_\tau, S_\nu) \bar{\Omega}_\nu$ (or $K(T_\tau, S_\nu)^T \bar{\Gamma}_\tau$) only for one $\Omega_\nu$ at a time, therefore the algorithm is memory efficient. As an example, \cref{fig:update_order} illustrates the procedure for constructing a 4-level hybrid factorization. Note that the random matrices for the inner factors are structured.

\subsection{Cost Analysis}
Let $c(n)$ denote the number of operations for the black-box multiplication of $K(S,T)\Omega$ or $K(T,S)^T\Omega$ for an arbitrary $n\times 1$ vector $\Omega$. In the best-case scenario, $c(n)=O(n\log n)$ as $K(T,S)$ is typically stored in a compressed form using $O(n\log n)$ storage units; in the worst-case scenario, $c(n)=O(n^2)$ when the matrix is explicitly stored in full. In what follows, we assume $c(n)=O(n\log n)$. Let $r=\max_{\tau,\nu}\{r_{\tau,\nu}\}$ denote the maximum butterfly rank. Here we analyze the computation and memory costs of \cref{alg:matvec_butterfly} when applied to two classes of butterfly-compressible matrices: (i) $r$ is constant (up to a logarithmic factor). (ii) $r=O(n^{1/4})$.       

\subsubsection{$r$ is constant}\label{sec:constr}
This case typically occurs when the matrix arises from the discretization of 2D surface integral equations exploiting strong or weak admissibility conditions \cite{Han_2013_butterflyLU,Yang_2017_HODBF,Yang_2020_BFprecondition}, 3D surface integral equation solvers using strong admissibility \cite{Han_2017_butterflyLUPEC}, low-dimensional Fourier operators \cite{li_butterfly_2015}, etc. For example, \cref{fig:2d3d}(a) shows the center-level partitioning of a 2D curve used in surface integral-based Helmholtz equation solvers in which all blocks have constant rank except for $O(1)$ ones with rank $O(\log n)$. (see \cite{Yang_2020_BFprecondition} for a proof). 

As described in \cref{sec:matvec}, there are $2^{l} (r+p)$ black-box matrix-vector multiplications by $K(T,S)^T$ at level $l=0,\ldots,l_m$ and $2^{(L-l)} (r+p)$ black-box matrix-vector multiplications by $K(T,S)$ at level $l=L,\ldots,l_m$. Therefore the black-box multiplications require a total of $2(r+p)(1+2+\ldots+2^{l_m})c(n)=O(rn^{1/2}c(n))=O(rn^{3/2}\log n)$ operations. It is worth noting that the multiplications on lines \ref{line:W_compute} and \ref{line:R_compute} only involve partial factors $V_{p_\tau, \nu_a}$ and $U_{\tau_a, p_\nu}$ and their computational cost is dominated by that of the black-box multiplications. In addition, the algorithm only stores multiplication results for each random matrix of dimensions $n\times(r+p)$ and the computed butterfly factors. The computation and memory costs of \cref{alg:matvec_butterfly} therefore scale as $O(n^{3/2}\log n)$ and $O(n\log n)$, respectively. 

\subsubsection{$r$ is $O(n^{1/4})$}\label{sec:incr}
This case often results from discretizing 3D surface integral equations using weak admissibility. For example, \cref{fig:2d3d}(b) shows the center-level partitioning of a 3D surface used in surface integral methods for Helmholtz equations. Out of the $16\times 16=256$ center-level blocks of size $O(n^{1/2})\times O(n^{1/2})$, only $4$ blocks have rank $O(n^{1/4})$ representing interactions between adjacent pairs. As the adjacent pair is typically co-planar, the edge shared by the pair can serve as the proxy surface that represent interactions between the pair. Therefore the rank is bounded by the edge dimension $O(n^{1/4})$. We claim, without proof, that except for $O(n^{1/4})$ blocks with rank at most $O(n^{1/4})$ at each level, the rest has constant rank. 

We only need to focus on the construction of blocks with non-constant ranks as the rest has been analyzed above. We first analyze the cost of black-box multiplications on lines \ref{line:matvecV} and \ref{line:matvecW}. For all blocks at level $0\leq l\leq l_m$, denote $N_l = \{\tau | r_{\tau,\nu}~non~constant\}$. It can be verified that $|N_l|=O(2^{\lfloor l/2\rfloor})$ and each $\tau \in N_l$ requires $O(r+p)=O(n^{1/4})$ matrix-vector multiplications. Therefore the black-box multiplications $K(T_\tau,S)^T\Gamma_\tau$ require a total of $\sum_{l=0}^{l_m}|N_l|O(n^{1/4})O(n\log n)=O(n^{1/2}c(n))=O(n^{3/2}\log n)$ operations. A similar number of operations is required for $K(T,S_\nu)\Omega_\nu$ on line \ref{line:matvecU} and \ref{line:matvecR}. In addition, since each computed factor has $O(n^{1/4})$ blocks of rank $O(n^{1/4})$, the storage cost is $O(n^{1/4})O(n^{1/2})L\leq O(n\log n)$. Overall the computation and memory costs of \cref{alg:matvec_butterfly} are of the same orders as those in the case of constant $r$. 

\begin{figure}[htbp!]
	\begin{center}
		\begin{tabular}{c}
			\includegraphics[height=1.5in]{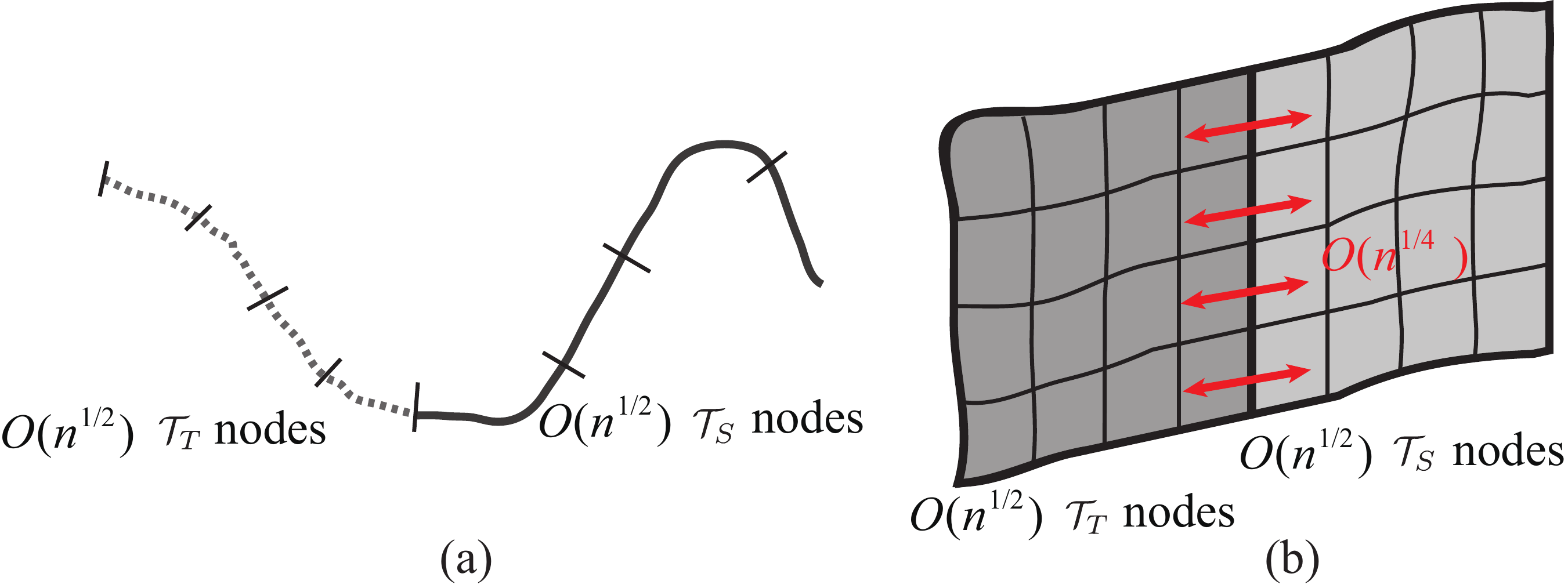}
		\end{tabular}
	\end{center}
	\caption{(a) Center-level geometrical partitioning of a 2D curve for a 4-level butterfly factorization of a weak-admissible matrix. The ranks for all center-level blocks are bounded by $O(\log n)$. (b) Center-level geometrical partitioning of a 3D surface for a 8-level butterfly factorization of a weak-admissible matrix. There are $O(n^{1/4})$ center-level blocks with rank $O(n^{1/4})$ (with the interaction pairs denoted by the red arrows)} 
	\label{fig:2d3d}
\end{figure}

\section{Error analysis}
In what follows, we assume that \cref{alg:matvec_random} used in \cref{alg:matvec_butterfly} computes a low-rank approximation satisfying $\|A - U U^T A\|_F \leqslant \epsilon \|A\|_F$ given a matrix $A$ (see \cite{halko_finding_2011}). 

\subsection{Approximation error for the nested basis $U_{\tau,\nu}$}
Given any level $l_m\leq l\leq L$ of $\mathcal{T}_T$, matrix $K(T, S)$ is partitioned into blocks $\{K(T_\tau, S_\nu)\}$
with $\tau$ at level $l$ of $\mathcal{T}_T$ and $\nu$ at level $(L-l)$ of $\mathcal{T}_S$, and $K(T_\tau, S_\nu)$ is approximated by $U_{\tau,\nu} U_{\tau,\nu}^T K(T_\tau, S_\nu)$ where $U_{\tau,\nu}$ is the nested basis computed in \cref{sec:compute_R}. \Cref{thm:column_wise} states that this approximation for all blocks at level $l$ has an error bounded by $\epsilon\sqrt{L-l+1}\|A\|_F$. 
\begin{theorem}\label{thm:column_wise}
	Given a level $l$ in the range $L\geqslant l \geqslant l_m$, the low-rank approximations of all the blocks $K(T_\tau, S_\nu)$ at level $l$ ($\tau$ at level $l$ of $\mathcal{T}_T$ and $\nu$ at level $(L-l)$ of $\mathcal{T}_S$) based on the computed column basis matrices $U_{\tau,\nu}$ have approximation errors bounded by
	\begin{equation}\label{eqn:thm_column}
	\sum_{\tau,\nu \text{ at level } l} \|K(T_\tau, S_\nu) - U_{\tau,\nu}U_{\tau,\nu}^T K(T_\tau, S_\nu)\|_F^2 \leqslant (L-l+1)\epsilon^2 \|K(T, S)\|_F^2,
	\end{equation}
	where the summation of $\tau$ is over the nodes at level $l$ of $\mathcal{T}_T$ and the summation of $\nu$ is over the nodes at level $(L-l)$ of $\mathcal{T}_S$. 
\end{theorem}

\begin{proof}
	We prove the theorem by induction for level $l$ changing from $L$ to $l_m$. 
	First, for $l = L$, the column basis matrix $U_{\tau,s}$ for block $K(T_\tau, S_s)$ at level $L$ of $\mathcal{T}_T$ is directly computed by \cref{alg:matvec_random}. 
	Thus, based on the above relative error threshold assumption, the approximation error is bounded by 
	\[
	\| K(T_\tau, S_s) - U_{\tau,s} U_{\tau,s}^T K(T_\tau, S_s)\|_F^2 \leqslant \epsilon^2\|K(T_\tau, S_s)\|_F^2. 
	\]
	Moreover, summing over all nodes $\tau$ on both sides of the above inequality proves \cref{eqn:thm_column} for $l = L$. 
	
	Assume that \cref{eqn:thm_column} holds true for level $(l+1)$. 
	For each node $\tau$ at level $l$ of $\mathcal{T}_T$, let $\nu_1$ and $\nu_2$ be siblings at level $(L-l)$ of $\mathcal{T}_S$. 
	Let $\{\tau_1, \tau_2\}$ be the children of $\tau$ at level $(l+1)$ of $\mathcal{T}_S$ and $\nu$ be the parent node of $\{\nu_1, \nu_2\}$ at level $(L-l-1)$ of $\mathcal{T}_S$. 
	Note that $K(T_\tau, S_{\nu_1})$ and $K(T_\tau, S_{\nu_2})$ are compressed blocks at level $l$ while $K(T_{\tau_1}, S_\nu)$ and $K(T_{\tau_2}, S_\nu)$ are compressed blocks at level $l+1$. 
	Moreover, these two sets of blocks correspond to the same large block in $K(T, S)$, i.e., 
	\[
	\begin{bmatrix}
	K(T_\tau, S_{\nu_1}) & K(T_\tau, S_{\nu_2})
	\end{bmatrix}
	= 
	\begin{bmatrix}
	K(T_{\tau_1}, S_\nu) \\
	K(T_{\tau_2}, S_\nu)
	\end{bmatrix}
	= 
	K(T_\tau, S_\nu).
	\]
	
	Recall that $U_{\tau, \nu_1}$ is computed in a nested fashion as 
	\begin{align*}
	K(T_\tau, S_{\nu_1})
	& = 
	\begin{bmatrix}
	K(T_{\tau_1}, S_{\nu_1}) \\
	K(T_{\tau_2}, S_{\nu_1})
	\end{bmatrix}
	\approx 
	\begin{bmatrix}
	U_{\tau_1, \nu} & \\
	& U_{\tau_2, \nu}
	\end{bmatrix}
	\begin{bmatrix}
	U_{\tau_1, \nu}^T K(T_{\tau_1}, S_{\nu_1}) \\
	U_{\tau_2, \nu}^T K(T_{\tau_2}, S_{\nu_1})
	\end{bmatrix}
	\\ & \approx
	\begin{bmatrix}
	U_{\tau_1, \nu} & \\
	& U_{\tau_2, \nu}
	\end{bmatrix}
	R_{\tau,\nu_1} R_{\tau, \nu_1}^T
	\begin{bmatrix}
	U_{\tau_1, \nu}^T K(T_{\tau_1}, S_{\nu_1}) \\
	U_{\tau_2, \nu}^T K(T_{\tau_2}, S_{\nu_1})
	\end{bmatrix}
	= 
	U_{\tau,\nu_1} U_{\tau,\nu_1}^T K(T_\tau, S_{\nu_1}),
	\end{align*}
	where $R_{\tau,\nu_1}$ is computed by applying \cref{alg:matvec_random} to the intermediate matrix above, i.e.,
	\begin{equation}\label{eqn:thm_approx}
	\begin{bmatrix}
	U_{\tau_1, \nu}^T K(T_{\tau_1}, S_{\nu_1}) \\
	U_{\tau_2, \nu}^T K(T_{\tau_2}, S_{\nu_1})
	\end{bmatrix}
	\approx 
	R_{\tau,\nu_1} R_{\tau, \nu_1}^T
	\begin{bmatrix}
	U_{\tau_1, \nu}^T K(T_{\tau_1}, S_{\nu_1}) \\
	U_{\tau_2, \nu}^T K(T_{\tau_2}, S_{\nu_1})
	\end{bmatrix}.
	\end{equation} 
	The approximation error of block $K(T_\tau, S_{\nu_1})$ at level $l$ can be estimated as
	\begin{align*}
	& \|K(T_\tau, S_{\nu_1}) - U_{\tau,\nu_1} U_{\tau,\nu_1}^T K(T_\tau, S_{\nu_1})\|_F^2
	\nonumber\\ & =
	\left\|K(T_\tau, S_{\nu_1}) - \begin{bmatrix}
	U_{\tau_1, \nu}& \\
	& U_{\tau_2, \nu}
	\end{bmatrix}
	\begin{bmatrix}
	U_{\tau_1, \nu}^T K(T_{\tau_1}, S_{\nu_1}) \\
	U_{\tau_2, \nu}^T K(T_{\tau_2}, S_{\nu_1})
	\end{bmatrix}\right\|_F^2  
	\nonumber\\ & \hspace{5em}+
	\left\|
	\begin{bmatrix}
	U_{\tau_1, \nu}& \\
	& U_{\tau_2, \nu}
	\end{bmatrix}
	\begin{bmatrix}
	U_{\tau_1, \nu}^T K(T_{\tau_1}, S_{\nu_1}) \\
	U_{\tau_2, \nu}^T K(T_{\tau_2}, S_{\nu_1})
	\end{bmatrix} - 
	U_{\tau,\nu_1} U_{\tau,\nu_1}^T K(T_\tau, S_{\nu_1})\right\|_F^2
	\nonumber\\ & = 
	\left\|K(T_{\tau_1}, S_{\nu_1}) \!-\! U_{\tau_1, \nu} U_{\tau_1, \nu}^TK(T_{\tau_1}, S_{\nu_1})\right\|_F^2
	+ \left\|K(T_{\tau_2}, S_{\nu_1}) - U_{\tau_2, \nu} U_{\tau_2, \nu}^TK(T_{\tau_2}, S_{\nu_1})\right\|_F^2
	\nonumber \\ & \hspace{5em} + 
	\left\|
	\begin{bmatrix}
	U_{\tau_1, \nu}^T K(T_{\tau_1}, S_{\nu_1}) \\
	U_{\tau_2, \nu}^T K(T_{\tau_2}, S_{\nu_1})
	\end{bmatrix} - 
	R_{\tau,\nu_1} R_{\tau, \nu_1}^T
	\begin{bmatrix}
	U_{\tau_1, \nu}^T K(T_{\tau_1}, S_{\nu_1}) \\
	U_{\tau_2, \nu}^T K(T_{\tau_2}, S_{\nu_1})
	\end{bmatrix}\right\|_F^2
	\end{align*}
	where the first equation is due to the orthogonality between the columns from the two bracketed terms, and the second equation follows from the unitary invariance of the Frobenius norm, as $U_{\tau_1, \nu}$ and $U_{\tau_2, \nu}$ both have orthonormal columns.
	The last term above, which is the approximation error of \cref{eqn:thm_approx}, is guaranteed to satisfy
	\[
	\epsilon^2
	\left\|\begin{bmatrix}
	U_{\tau_1, \nu}^T K(T_{\tau_1}, S_{\nu_1}) \\
	U_{\tau_2, \nu}^T K(T_{\tau_2}, S_{\nu_1})
	\end{bmatrix}\right\|_F^2
	\leqslant 
	\epsilon^2
	\left\|\begin{bmatrix}
	K(T_{\tau_1}, S_{\nu_1}) \\
	K(T_{\tau_2}, S_{\nu_1})
	\end{bmatrix}\right\|_F^2
	= \epsilon^2 \|K(T_\tau, S_{\nu_1})\|_F^2.
	\]
	Similarly, we can estimate the approximation error of block $K(T_\tau, S_{\nu_2})$ at level $l$ as
	\begin{align*}
	& \|K(T_\tau, S_{\nu_2}) - U_{\tau,\nu_2} U_{\tau,\nu_2}^T K(T_\tau, S_{\nu_2})\|_F^2
	\\ & \leqslant
	\left\|K(T_{\tau_1}, S_{\nu_2}) - U_{\tau_1, \nu} U_{\tau_1, \nu}^TK(T_{\tau_1}, S_{\nu_2})\right\|_F^2
	\\ & \hspace{3em} + \left\|K(T_{\tau_2}, S_{\nu_2}) - U_{\tau_2, \nu} U_{\tau_2, \nu}^TK(T_{\tau_2}, S_{\nu_2})\right\|_F^2+ 
	\epsilon^2 \|K(T_\tau, S_{\nu_2})\|_F^2.
	\end{align*}
	
	Assembling the above three inequalities and using $S_\nu = S_{\nu_1}\cup S_{\nu_2}$, we obtain
	\begin{align*}
	& \|K(T_\tau, S_{\nu_1}) - U_{\tau,\nu_1} U_{\tau,\nu_1}^T K(T_\tau, S_{\nu_1})\|_F^2 + \|K(T_\tau, S_{\nu_2}) - U_{\tau,\nu_2} U_{\tau,\nu_2}^T K(T_\tau, S_{\nu_2})\|_F^2
	\nonumber \\ & \leqslant
	\left\|K(T_{\tau_1}, S_\nu) - U_{\tau_1, \nu} U_{\tau_1, \nu}^TK(T_{\tau_1}, S_\nu)\right\|_F^2
	+ \left\|K(T_{\tau_2}, S_\nu) - U_{\tau_2, \nu} U_{\tau_2, \nu}^TK(T_{\tau_2}, S_\nu)\right\|_F^2
	\nonumber \\ & \hspace{3em}
	+ \epsilon \left(\|K(T_\tau, S_{\nu_1})\|_F^2 + \|K(T_\tau, S_{\nu_2})\|_F^2\right).
	\end{align*}
	The left hand side of the inequality is the approximation error of the two blocks at level $l$ while the first two terms on the right hand side are the approximation errors of two blocks at level $(l+1)$.
	Summing over all the nodes $\tau$ at level $l$ of $\mathcal{T}_T$ and all the node pairs $(\nu_1, \nu_2)$ at level $(L-l)$ of $\mathcal{T}_S$ on both sides of this inequality, we obtain, 
	\begin{align*}
	\hspace{2em} & \sum_{\tau, \nu \text{ at level } l} \|K(T_\tau, S_\nu) - U_{\tau,\nu}U_{\tau,\nu}^T K(T_\tau, S_\nu)\|_F^2
	\\ & \leqslant
	\sum_{\tau, \nu \text{ at level } (l+1)} \|K(T_\tau, S_\nu) - U_{\tau,\nu}U_{\tau,\nu}^T K(T_\tau, S_\nu)\|_F^2 + \epsilon^2 \|K(T, S)\|_F^2
	\\ & \leqslant (L-(l+1)+1) \epsilon^2 \|K(T, S)\|_F^2 + \epsilon^2 \|K(T, S)\|_F^2 = (L-l+1)\epsilon^2\|K(T, S)\|_F^2.
	\end{align*}
	
\end{proof}

\subsection{Approximation error for the nested basis $V_{\tau,\nu}$}
The overall approximation error for the projection to the row bases is bounded just like that of the column bases shown in \Cref{thm:row_wise}.
\begin{theorem}\label{thm:row_wise}
	Given a level l in the range $L\geqslant l \geqslant l_m$, the low-rank approximations of all the blocks $K(T_\tau, S_\nu)$ at level $l$ ($\tau$ at level $l$ of $\mathcal{T}_T$ and $\nu$ at level $(L-l)$ of $\mathcal{T}_S$) based on the computed row basis matrices $V_{\tau,\nu}$ have approximation errors bounded as
	\begin{equation}\label{eqn:thm_row}
	\sum_{\tau,\nu \text{ at level } l} \|K(T_\tau, S_\nu) - K(T_\tau, S_\nu) V_{\tau,\nu}V_{\tau,\nu}^T\|_F^2 \leqslant (l+1)\epsilon^2 \|K(T, S)\|_F^2,
	\end{equation}
	where the summation of $\tau$ is over the nodes at level $l$ of $\mathcal{T}_T$ and the summation of $\nu$ is over the nodes at level $(L-l)$ of $\mathcal{T}_S$. 
\end{theorem}

\subsection{Approximation error for the overall factorization}
Combining the above two error analyses for the column- and row-wise butterfly factorizations, the overall approximation error of a hybrid butterfly factorization at any level $l$ can be bounded as shown in \cref{thm:hybrid} 

\begin{theorem}\label{thm:hybrid}
	Given center level $l_m=\lfloor L/2\rfloor$, the low-rank approximations of all the blocks $K(T_\tau, S_\nu)$ at level $l_m$ ($\tau$ at level $l_m$ of $\mathcal{T}_T$ and $\nu$ at level $(L-l_m)$ of $\mathcal{T}_S$) based on the computed basis matrices $U_{\tau,\nu}$ and $V_{\tau,\nu}$ have an overall approximation error bounded by
	\begin{equation}\label{eqn:thm_hybrid}
	\sum_{\tau,\nu \text{ at level } l} \|K(T_\tau, S_\nu) - U_{\tau,\nu}U_{\tau,\nu}^T K(T_\tau, S_\nu) V_{\tau,\nu}V_{\tau,\nu}^T\|_F^2 \leqslant (L+2)\epsilon^2 \|K(T, S)\|_F^2,
	\end{equation}
	where the summation of $\tau$ is over the nodes at level $l_m$ of $\mathcal{T}_T$ and the summation of $\nu$ is over the nodes at level $(L-l_m)$ of $\mathcal{T}_S$. 
\end{theorem}

\begin{proof}
	Every block $K(T_\tau, S_\nu)$ has its approximation error bounded as
	\begin{align*}
	& \|K(T_\tau, S_\nu) - U_{\tau,\nu}U_{\tau,\nu}^T K(T_\tau, S_\nu) V_{\tau,\nu}V_{\tau,\nu}^T\|_F^2 
	\\ &\leqslant
	\|K(T_\tau, S_\nu)\! - U_{\tau,\nu}U_{\tau,\nu}^T K(T_\tau, S_\nu)\|_F^2
	\\ &\hspace{3em} + \|U_{\tau,\nu}U_{\tau,\nu}^TK(T_\tau, S_\nu)\! - U_{\tau,\nu}U_{\tau,\nu}^T K(T_\tau, S_\nu) V_{\tau,\nu}V_{\tau,\nu}^T\|_F^2 
	\\ &\leqslant \|K(T_\tau, S_\nu) - U_{\tau,\nu}U_{\tau,\nu}^T K(T_\tau, S_\nu)\|_F^2 + 
	\|K(T_\tau, S_\nu) - K(T_\tau, S_\nu) V_{\tau,\nu}V_{\tau,\nu}^T\|_F^2. 
	\end{align*}
	Using \cref{eqn:thm_column,eqn:thm_row} in the above equation, inequality \cref{eqn:thm_hybrid} is proven. 
\end{proof}

\section{Parallelization}
This section outlines a distributed-memory implementation of \cref{alg:matvec_butterfly}. We first consider the task of parallelizing a butterfly-vector multiplication assuming that the butterfly factors are stored in a distributed fashion. Note that this is the dominant computational task in the proposed algorithm as the black-box multiplications on lines \ref{line:matvecV}, \ref{line:matvecW}, \ref{line:matvecU}, \ref{line:matvecR} often involve existing butterfly representations, and the explicit multiplications on lines \ref{line:matvecHalf}, \ref{line:W_compute}, \ref{line:R_compute} involve partial butterfly factors. A good parallelization strategy for these two types of multiplications therefore is paramount to the parallel implementation of the proposed randomized butterfly reconstruction scheme.

Without loss of generality, we assume that $r_{\tau,\nu}=r$ for some constant $r$, the number of butterfly levels $L$ is even, $n=r2^L$, and the number of processes $1\leq p\leq 2^L$ is a power of two. To estimate the algorithm's communication cost, we only analyze the number of messages and communication volume as we assume the time to communicate a message of size $m$ between two processes is $\alpha+\beta m$ where $\alpha$ and $\beta$ represent message latency and inverse bandwidth, respectively \cite{Hockney_1994_comm}.   

\subsection{Column/Row-wise factorization} We first describe the parallelization scheme for the column-wise butterfly factorization studied in \cite{Jack_2014_Parallel}. The parallel data layout can be descried as follows: (i) starting from level $L$ of $\mathcal{T}_T$, one process stores $2^L/p$ consecutive blocks $U_{\tau, s}$ following a 1D-row layout. (ii) At levels $l=L-1,\ldots,1$, let $\{\tau_1, \tau_2\}$ be the children of $\tau$ at level $l$ of $\mathcal{T}_T$ and $\{\nu_1, \nu_2\}$ be the children of $\nu$ at level $L-l-1$ of $\mathcal{T}_S$. Consider the \textit{combined} transfer matrix:
\begin{align}
R_{\tau,\nu}=
\begin{bmatrix}
R_{\tau_1,\nu} & \\
& R_{\tau_2,\nu} 
\end{bmatrix}
\begin{bmatrix}
R_{\tau, \nu_1}&\!\!\!\!R_{\tau, \nu_1}
\end{bmatrix}
\end{align} 
where $R_{\tau_a,\nu}$ can be replaced by $U_{\tau_a,\nu}$ if $l=L-1$. The parallelization scheme stores $R_{\tau_a,\nu}$ (or $U_{\tau_a,\nu}$ if $l=L-1$) and $R_{\tau,\nu_a}$ on the same process. (iii) At level $l=0$, the blocks $E_{t,\nu}$ and $R_{t,\nu}$ are stored on the same process. \cref{fig:parallel}(a) illustrates the data layout for a 4-level column-wise butterfly factorization using 4 processes. Note that at level $l=0$, the layout of $E_{t,\nu}$ is similar to the 1D-row layout for $U_{\tau, s}$, but in an index-reversed order.

When multiplying a (partial) butterfly stored as described above by a vector stored using the 1D-row layout, an all-to-all communication is required to convert the vector from the 1D-row layout to the index-reversed layout of $E_{t,\nu}$. To this end, each process communicates $\min\{{2^L}/{p},p-1\}$ messages of total size $r2^L/p$. For the example in \cref{fig:parallel}(a), process 0 (light green) needs to scatter the locally stored part of $\Omega$ of size $2^L/p\times r$ to all other processes before the multiplication operation with $E^0$ can take place. After that, there are $\log p$ levels requiring pair-wise exchanges of the intermediate multiplication results $R^l\ldots R^0E^0\Omega$ (e.g., after multiplication with $R_0$ and $R_1$ in \cref{fig:parallel}(a)). For each exchange operation, reductions involving messages of size $r2^L/p$ between two processes are performed (note: broadcast is conducted for multiplying the transpose $K(T,S)^T$). For example, considering the multiplication with $R_{\tau,\nu_1}$ (stored on process 0 in light green) and $R_{\tau,\nu_2}$ (stored on process 2 in dark green) in the first diagonal block of $R^0$, the local multiplication results of size $2r\times r$ require a reduction between processes 0 and 2 before the local multiplication with the first diagonal block of $R^1$ is performed. The communication costs are listed in \cref{tab:comm}. 
   
The parallelization of the row-wise butterfly factorization can be described similarly: At level $l=0$ of $\mathcal{T}_0$, $V_{t,\nu}$ is stored using the 1D-row layout; at levels $l=1,\ldots,L$, $W_{\tau,\nu_a}$ (or $V_{\tau,\nu_a}$ if $l=1$) and $W_{\tau_a,\nu}$ are stored on the same process. The communication costs for the row-wise parallelization are the same as those for the column-wise parallelization. 

\subsection{Hybrid factorization} Since \cref{alg:matvec_butterfly} computes a hybrid factorization, it is more convenient and efficient to combine the column-wise and row-wise parallel data layouts. Specifically, the factors $U^LR^{L-1}R^{L-2}\ldots R^{l_m}$ are stored using the column-wise layout whereas those of $W^{l_m}W^{l_m-1}\ldots W^1V^0$ adhere to the row-wise layout. The block $B_{\tau,\nu}$ is handled by the same process as $W_{\tau,\nu}$ in $W^{l_m}$. When multiplying the (partial) butterfly with a vector $\Omega$, all-to-all communication is needed for the intermediate result $B^{l_m}W^{l_m}W^{l_m-1}\ldots W^1V^0\Omega$. In contrast to the column/row-wise layout, the number of levels requiring pair-wise exchange is only $\max\{0,\log {p^2}/{2^L}\}$. Note that no exchange is needed if $p\leq 2^{L/2}$. As an example, \cref{fig:parallel}(b) shows the multiplication of a 4-level hybrid butterfly using 4 processes. Here, $p= 2^{L/2}$ and the only communication is the all-to-all operation that occurs after multiplying with $B^2$ to switch from the row- to the column-wise layout. \cref{tab:comm} lists the communication costs for the hybrid data layout. Clearly, the multiplication using the hybrid factorization requires less communication than multiplication with the column/row-wise factorization. 

Now we can summarize the proposed parallelization strategy for \cref{alg:matvec_butterfly} as follows:
\begin{itemize}
	\item The black-box multiplications on lines \ref{line:matvecV}, \ref{line:matvecW}, \ref{line:matvecU}, \ref{line:matvecR} follow the hybrid layout if they involve existing parallel butterfly representations. 
	\item The computed butterfly factors follow the hybrid layout for each process that applies \cref{alg:matvec_random} on lines \ref{line:W_compute} and \ref{line:R_compute} locally for the blocks it is in charge of. 
	\item The explicit multiplications with computed butterfly blocks on lines \ref{line:matvecHalf}, \ref{line:W_compute}, \ref{line:R_compute} follow the hybrid layout but may not involve all the processes.    
\end{itemize}

\begin{figure}[htbp!]
	\begin{center}
		\begin{tabular}{c}
			\includegraphics[height=2in]{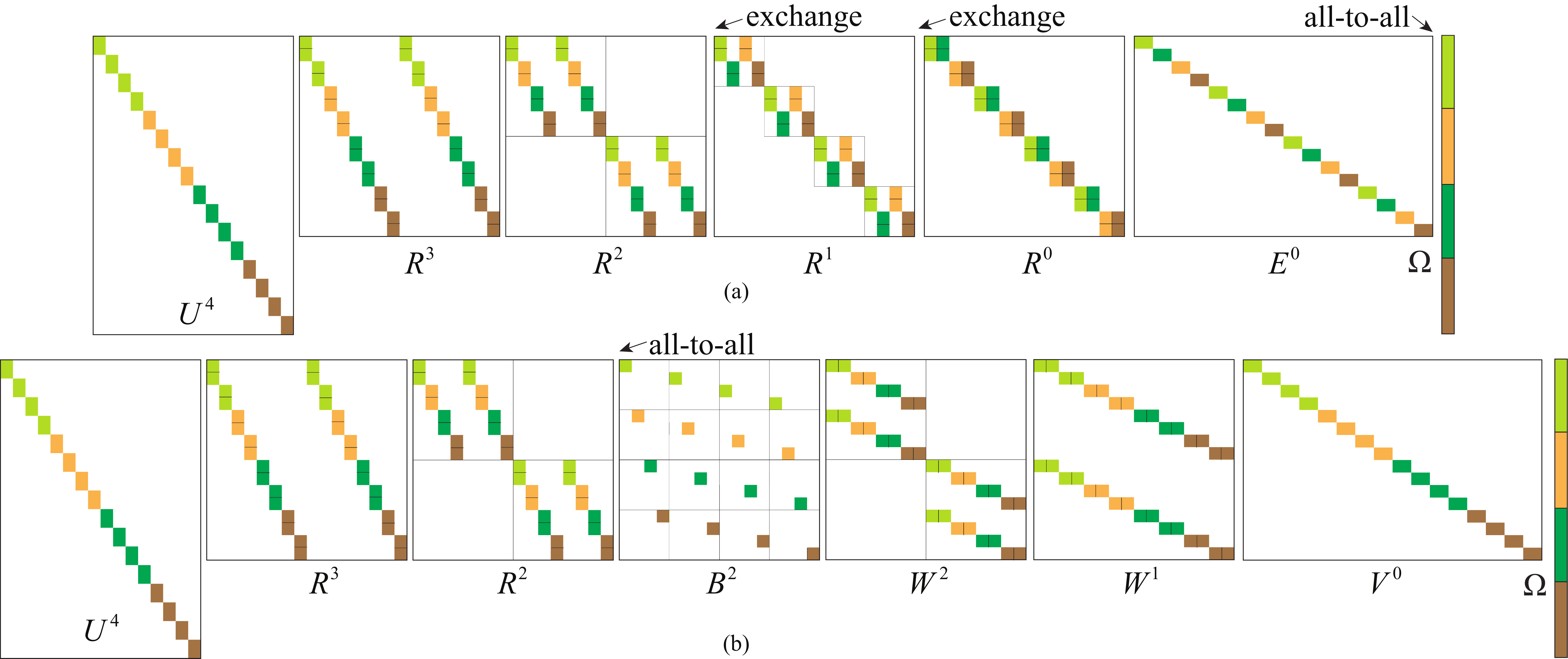}
		\end{tabular}
	\end{center}
	\caption{Parallel data layout for distributing a 4-level (a) column-wise and (b) hybrid factorization with 4 processes. Each color represents one process. The arrows denote places where exchange and all-to-all communications are needed for multiplying the butterfly with a 1D row distributed vector. Note that no exchange is needed for the hybrid butterfly.} 
	\label{fig:parallel}
\end{figure}

\begin{table*}[!tp]
	\centering	
	\begin{tabular}{|c|c|c|c|c|}
		\hline
		 & \multicolumn{2}{c|}{exchange}  & \multicolumn{2}{c|}{all-to-all} \\		
		\hline
		&volume & message count & volume & message count \\
		\hline		
		column/row	&$\frac{r2^L\log p}{p}$	& $\log p$ & $\frac{r2^L}{p}$ & $\min\{\frac{2^L}{p},p-1\}$ \\
		\hline
		hybrid &$\frac{r2^L}{p}\max\{0,\log \frac{p^2}{2^L}\}$	& $\max\{0,\log \frac{p^2}{2^L}\}$ & $\frac{r2^L}{p}$ & $\min\{\frac{2^L}{p},p-1\}$ \\ 
		\hline
	\end{tabular}
	\caption{Communication volume and message counts for one matrix-vector multiplication.}\label{tab:comm}
\end{table*}

\section{Numerical results}
This section provides several examples to demonstrate the accuracy and efficiency of the proposed randomized algorithm. The accuracy of the proposed algorithms is characterized by 
\begin{align}\label{eqn:error}
error=\frac{\left\lVert A\Omega-\big(U^LR^{L-1}R^{L-2}\ldots R^{l}\big)B^{l}\big(W^lW^{l-1}\ldots W^1V^0\big)\Omega \right\rVert_F}{\left\lVert A\Omega \right\rVert_F}
\end{align}
with a random testing matrix $\Omega$ of 16 columns. All experiments are performed on the Cori Haswell machine at NERSC, which is a Cray XC40 system and consists of 2388 dual-socket nodes with Intel Xeon E5-2698v3 processors running 16 cores per socket. The nodes are configured with 128 GB of DDR4 memory clocked at 2133 MHz. Unless stated otherwise, all experiments use one Cori node. 

\subsection{Exact butterfly factorization}
We first apply the algorithm to ``recover" a matrix with known butterfly factorization. In what follows, we use the symbol ``$\bar{a}$" to differentiate the blocks and factors of the known butterfly from the computed ones. Let $A$ be a $n\times n$ matrix with a given $L$-level butterfly factorization $A=(\bar{U}^L\bar{R}^{L-1}\bar{R}^{L-2}\ldots \bar{R}^{l})\bar{B}^{l}(\bar{W}^l\bar{W}^{l-1}\ldots \bar{W}^1\bar{V}^0)$ that has $r_{\tau,\nu}=r$ for all blocks at all levels for some constant $r$. The blocks $\bar{W}_{\tau,\nu}$, $\bar{R}_{\tau,\nu}$, and $\bar{B}_{\tau,\nu}$ have dimensions $r\times2r$, $2r\times r$ and $r\times r$, respectively. We choose the matrix dimension $n=2^{L+3} $ such that $\bar{U}_{\tau,s}$ and $\bar{V}_{t,\nu}$ have dimensions $8\times r$. The blocks $\bar{U}_{\tau,s}$, $\bar{V}_{t,\nu}$, $\bar{W}_{\tau,\nu}$, $\bar{R}_{\tau,\nu}$ are constructed as random unitary matrices, and the entries of $\bar{B}_{\tau,\nu}$ are random variables which are independent and identically distributed, following a normal distribution. We use this explicit representation to perform ``black-box" matrix-vector multiplications and apply \cref{alg:matvec_butterfly} to retrieve the butterfly factorization. 

The memory costs for varying $n$ with fixed $r=8$ when using the proposed algorithm and the randomized SVD-based reference algorithm in \cite{li_butterfly_2015} are plotted in Figure \ref{fig:complexity_random}(a). We set $p=2$ and $r_0=4$ in \cref{alg:matvec_butterfly}. The proposed algorithm requires only $O(n{\rm log}n)$ memory as opposed to the reference algorithm, which requires $O(n^{1.5} )$ memory. 

Next, the performance of the proposed parallelization scheme is demonstrated by applying \cref{alg:matvec_butterfly} to the butterfly representation with $n=\num{2.56E+6}$ and $r=10$. \cref{fig:complexity_random}(b) plots the runtime of \cref{alg:matvec_butterfly} for process counts 16 to 2048. The runtime for a single matrix-vector multiplication when the input butterfly representation follows hybrid and column-wise data layouts is also shown. As predicted by \cref{tab:comm}, the hybrid data layout yields a substantially lower communication cost.
\begin{figure}[htbp!]
	\begin{center}
		\begin{tabular}{c c}
			\includegraphics[height=1.5in]{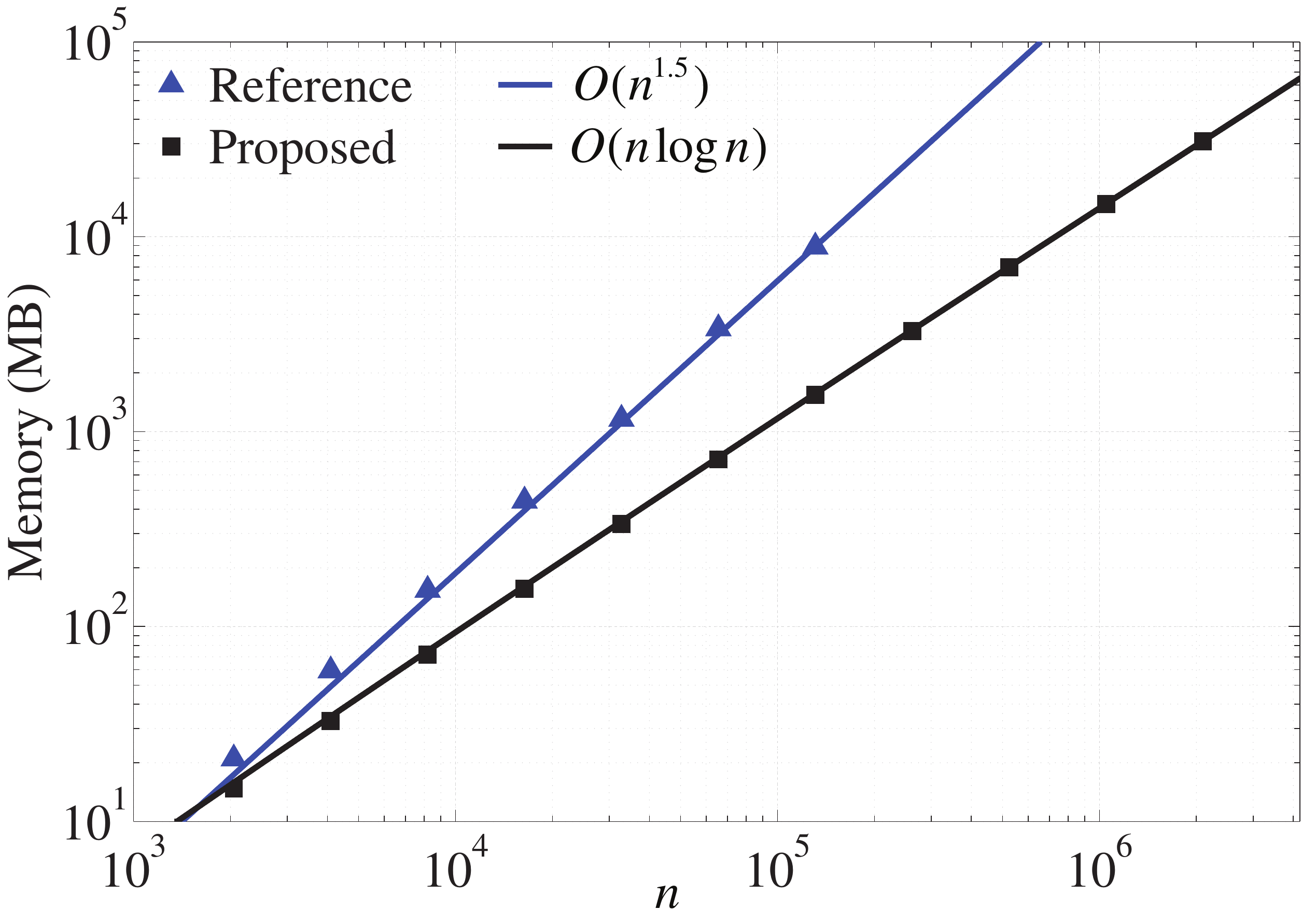} &\includegraphics[height=1.5in]{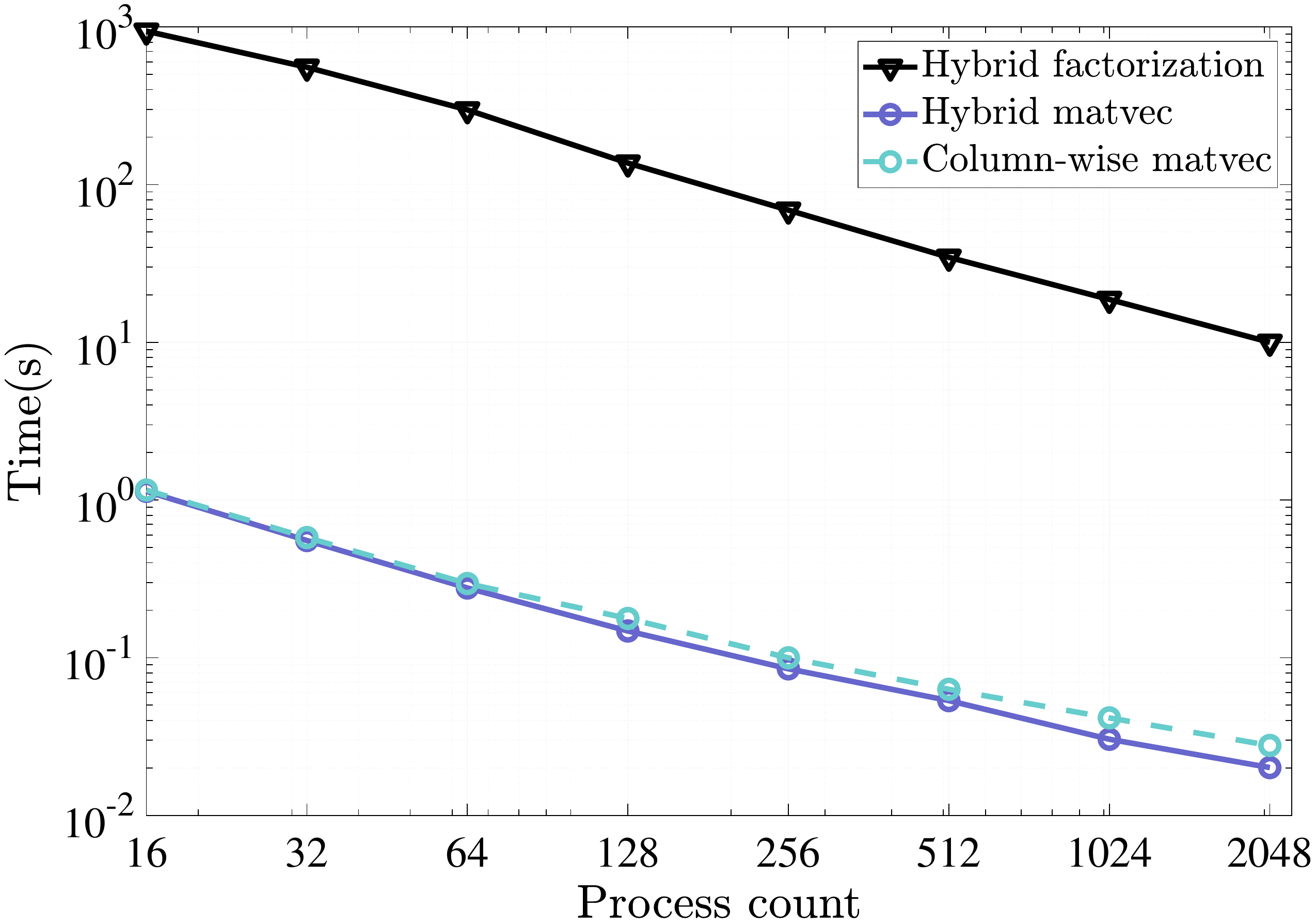}
		\end{tabular}
	\end{center}
	\caption{(a) Memory for applying \cref{alg:matvec_butterfly} and the reference algorithms to an exact butterfly representation. (b) Runtime for applying \cref{alg:matvec_butterfly} and matrix-vector multiplications to the exact butterfly representation for varying process counts.} 
	\label{fig:complexity_random}
\end{figure}


\subsection{2D Helmholtz kernel}

\noindent Next, consider the following wave scattering example. Let $C^{1} $ and $C^{2} $ denote two disjoint curves. Suppose $C^{1} $ and $C^{2} $ are partitioned into $n$ and $m$ constant-sized segments $C_{i}^{1} $ and $C_{j}^{2} $, $i=1,...,n$, $j=1,...,m$. Let $k_{0} $ denote the wavenumber, then $m,n=O(k_{0} )$. Consider the following $m\times n$ matrix $A$ 
\begin{equation} \label{eqn:smat} 
A=Z^{21} (Z^{11} )^{-1}  
\end{equation} 
\begin{equation} \label{5.3)} 
Z_{i,j}^{11} =\int _{C_{j}^{1} } H_{0}^{(2)} (k_{0} |\rho _{i}^{1} -\rho |)d\rho ,\; i,j=1,...,n 
\end{equation} 
\begin{equation} \label{5.4)} 
Z_{i,j}^{21} =\int _{C_{j}^{1} } H_{0}^{(2)} (k_{0} |\rho _{i}^{2} -\rho |)d\rho ,\; i=1,...,m,\; j=1,...,n 
\end{equation} 
Here, $H_{0}^{2} (\cdot )$ is the zeroth-order Hankel function of the second kind, $\rho $ is a position vector on $C^{1} $, $\rho _{i}^{1} $ and $\rho _{i}^{2} $ denote the center of segment $C_{i}^{1} $ and $C_{i}^{2} $. In principle, the $n\times n$ matrix $(Z^{11} )^{-1} $ relates the equivalent source on $C^{1} $ to the incident fields on $C^{1} $, and the $m\times n$ matrix $Z^{21} $ computes fields observed on $C^{2} $ scattered by the source on $C^{1} $. The matrix $A$ in (\ref{eqn:smat}) resembles the \textit{scattering matrix} as it relates the incident fields on $C^{1} $ to its scattered fields on $C^{2} $ \cite{Hao_2015_scatteringmatrix,Yang_2017_HODBF}. In what follows, we seek a butterfly-compressed representation of the matrix $A$. 

In this example, suppose $C^{1} $ and $C^{2} $ are two parallel lines with length $L$ and their respective distance $L$. It is well known that $Z^{11} $ and its inverse have compressed representations using $\mathcal{H}$-matrix or other low-rank factorization-based hierarchical techniques, requiring at most $O(n{\rm log}n)$ computation and memory resources \cite{Chandrasekaran_2007_HSS,Hackbusch_2003_Hmatrix,martinsson_fast_2005}. In addition, $Z^{21} $ can be compressed by the butterfly factorization requiring $O(n{\rm log}n)$ computation and memory resources \cite{Yang_2017_HODBF,Yang_2020_BFprecondition}. Therefore, $A$ and its transpose can be applied to any vector in $O(n{\rm log}n)$ operations irrespective of wavenumber $k_{0} $.

To compute a butterfly-compressed representation of $A$ using the proposed algorithm, the lengths of line segments $C_{i}^{1} $ and $C_{i}^{2} $ are set to approximately $0.05\lambda $ with $\lambda =2\pi /k_{0} $ denoting the wavelength. The sizes of the leaf-level point sets $T_{\tau} $ and $S_{\nu}$ are set to approximately 39. The matrices $(Z^{11} )^{-1} $ and $Z^{21} $ are compressed respectively with the $\mathcal{H}$-matrix and butterfly factorization with a high accuracy (tolerance \num{E-8}), respectively. Note that these compressed representations are used instead of $A$ in \cref{eqn:error} and in black-box multiplications. 
\begin{table*}[!tp]
	\centering	
	\begin{tabular}{|c|c|c|c|c|c|c|}
		\hline
		$n$ & $L$ &$\epsilon$ & $r$ & error & Time (sec) & Memory (MB)\\		
		\hline
		\hline
		20000	&9	&1E-03 &10	&3.63E-04 	 &3.61E+01	 &1.68E+01 \\
		\hline
		80000	&11	&1E-03 &10	 &4.22E-04 	 &3.27E+02	 &7.71E+01 \\
		\hline
		320000	&13	&1E-03 &10	 &4.98E-04   &3.27E+03	 &3.47E+02 \\
		\hline
		1280000	&15	&1E-03 &10	 &7.38E-04 	 &3.02E+04	 &1.58E+03 \\
		\hline
		\hline
		20000	&9	&1E-04 &12	 &2.45E-05 	 &7.80E+01	 &3.13E+01 \\
		\hline
		80000	&11	&1E-04 &12	 &2.39E-05 	 &7.21E+02	 &1.32E+02 \\
		\hline
		320000	&13	&1E-04 &12	 &2.37E-05 	 &6.46E+03	 &5.99E+02 \\
		\hline
		1280000	&15	&1E-04 &12	 &3.26E-05 	 &5.87E+04	 &2.64E+03 \\
		\hline
		\hline
		20000	&9	&1E-05 &14	 &8.09E-06 	 &8.10E+01	 &3.71E+01 \\
		\hline
		80000	&11	&1E-05 &14	 &8.39E-06 	 &7.82E+02	 &1.71E+02 \\
		\hline
		320000	&13	&1E-05 &14	 &9.11E-06 	 &7.17E+03	 &7.92E+02 \\
		\hline
		1280000	&15	&1E-05 &14	 &9.57E-06 	 &6.17E+04	 &3.56E+03 \\
		\hline
	\end{tabular}
	\caption{Time, memory and measured error for computing a hybrid factorization of \cref{eqn:smat} using the proposed algorithm with varying matrix size $n$ and tolerance $\epsilon$.}\label{tab:Tmat}
\end{table*} 

\begin{figure}[htbp!]
	\begin{center}
		\begin{tabular}{c c}
			\includegraphics[height=1.5in]{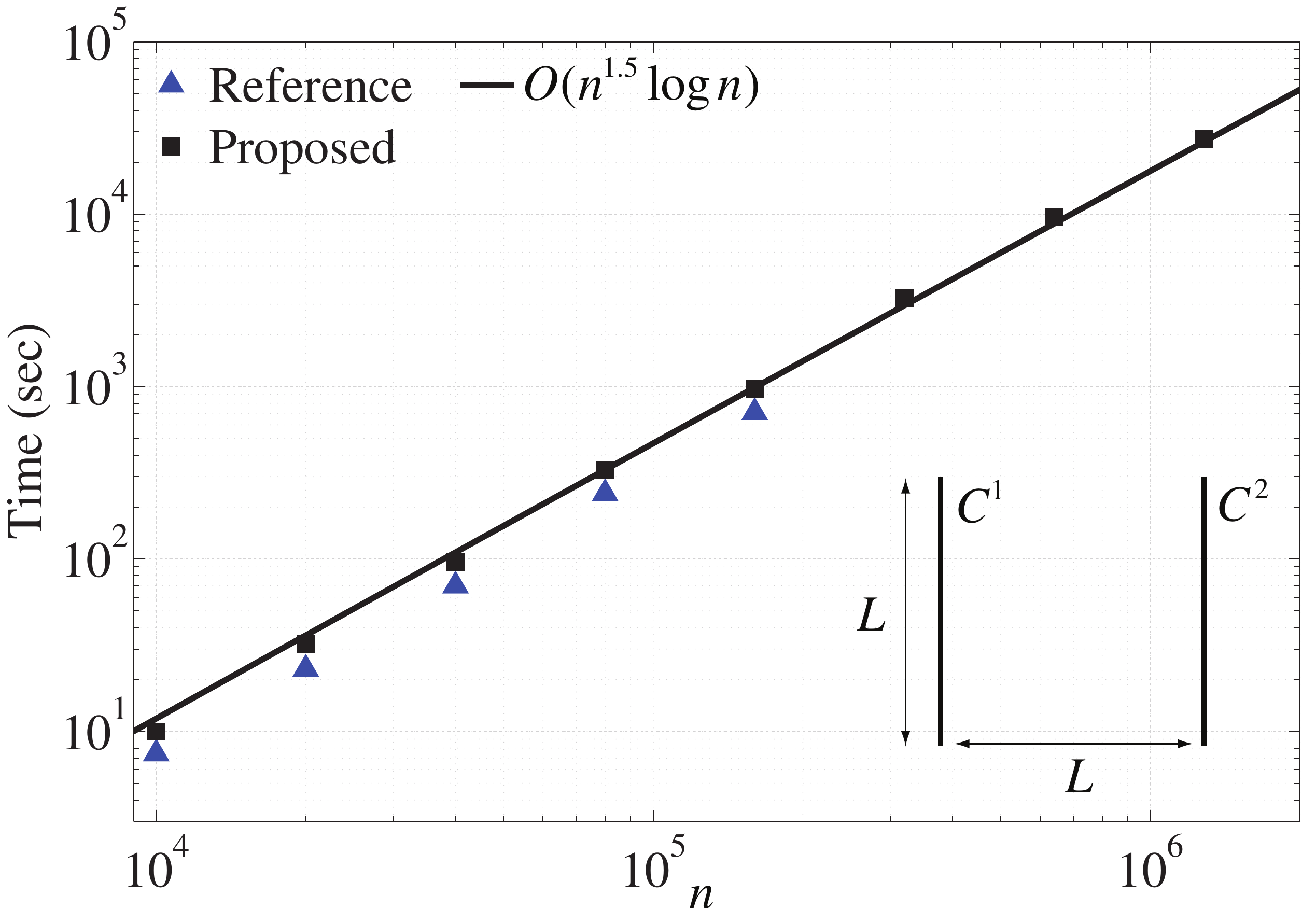} &\includegraphics[height=1.5in]{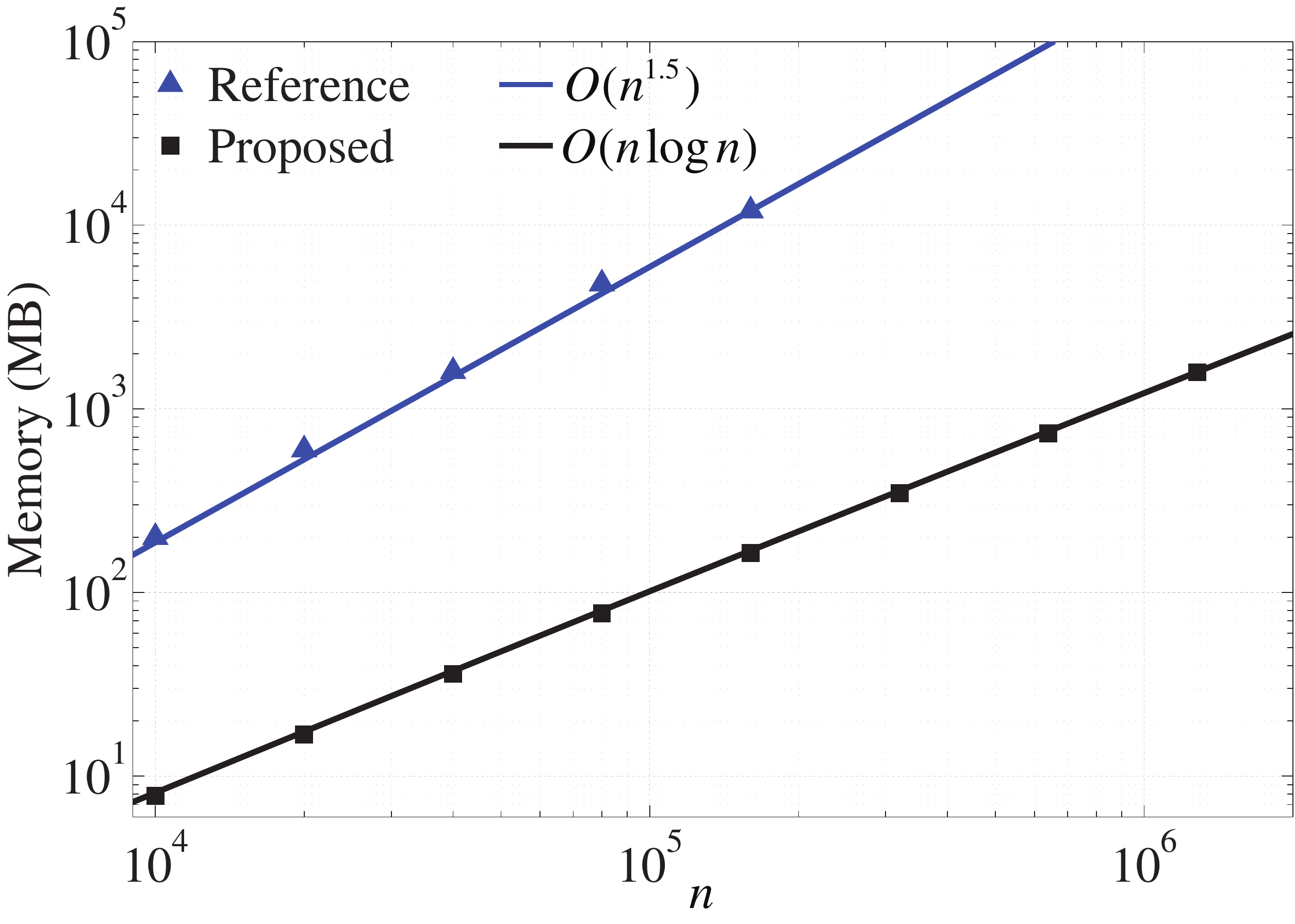}
		\end{tabular}
	\end{center}
	\caption{(a) Computation time and (b) memory for applying \cref{alg:matvec_butterfly} and the reference algorithms to \cref{eqn:smat}. } 
	\label{fig:complexity_Tmat}
\end{figure}
 
The computational results with butterfly level $L=9,11,13,15$ for different tolerances $\epsilon=\num{e-3},\num{e-4},\num{e-5}$ are listed in \cref{tab:Tmat}. We set $p=2$ and $r_0=4$ in \cref{alg:matvec_butterfly}. The measured error scales proportional to $L$ as predicted in \cref{thm:hybrid}. For $\epsilon=\num{E-3}$, the computationally most expensive case (when $n=\num{1.28E6}$) requires about 9 h CPU time and 1.6 GB memory. Moreover, the memory cost and the factorization time obey the predicted $O(n{\rm log}n)$ and $O(n^{1.5} {\rm log}n)$ scaling estimates.

The computation time and memory cost with $l=8,...,15$ and $\epsilon=\num{E-3}$ are plotted in \cref{fig:complexity_Tmat} using the proposed algorithm and the reference algorithm in [10]. The proposed algorithm is slightly slower than the reference algorithm as it requires slightly more testing vectors to reduce the memory cost. That said, the proposed algorithm requires much less memory than the reference algorithm (see \cref{fig:complexity_Tmat}(b)).

\subsection{3D Helmholtz kernel}
\noindent Finally, we consider the wave interactions between two semi-sphere surfaces $C^{1}$ and $C^{2} $ of unit radius adjacent to each other. Each semi-sphere is discretized via the Nyström method into $n$ sample points. We seek a butterfly factorization of the following $n\times n$ matrix $A$:
\begin{align}
A_{i,j} = \frac{e^{i2\pi \kappa |\rho_i-\rho_j|}}{|\rho_i-\rho_j|}\label{eqn:3dkernel} 
\end{align} 
with $\rho_i$ and $\rho_j$ denoting sample point $i$ and $j$ on $C^{1}$ and $C^{2}$, respectively. The wavenumber $\kappa$ is set such that $n=50\kappa^2/\pi$ represents approximately 10 sample points per wavelength. We first compute $A\approx(\bar{U}^L\bar{R}^{L-1}\bar{R}^{L-2}\ldots \bar{R}^{l})\bar{B}^{l}(\bar{W}^l\bar{W}^{l-1}\ldots \bar{W}^1\bar{V}^0)$ with $\epsilon=\num{E-06}$ and use the result for \cref{eqn:error} and the black-box multiplications.

The computational results with butterfly level $L=8,10,12,14$ for different tolerances $\epsilon=\num{e-2},\num{e-3},\num{e-4}$ are listed in \cref{tab:EM3D}. We set $p=4$ and $r_0=64$ in \cref{alg:matvec_butterfly}. These experiments use 2 Cori nodes with a total of 64 MPI processes. The proposed algorithm achieves the desired accuracies predicted by \cref{thm:hybrid}. The computation time, memory and observed rank with $\epsilon=\num{E-2}$ are plotted in \cref{fig:complexity_EM3D} using the proposed algorithm. Despite the $O(n^{0.25})$ rank scaling (see \cref{fig:2d3d}(b) as an illustration), \cref{alg:matvec_butterfly} still attains $O(n^{1.5}{\rm log}n)$ computation and $O(n{\rm log}n)$ memory complexities as estimated in \cref{sec:incr}. In addition, the estimated memory usage with the reference algorithm  is also plotted in \cref{fig:complexity_EM3D}(b), the proposed algorithm requires much less memory in comparison.

\section{Conclusion and discussion}
\label{sec:conclusion}
This paper presented a fast and memory-efficient randomized algorithm for computing the butterfly factorization of a matrix assuming the availability of a black-box algorithm for applying the matrix and its transpose to a vector. The proposed algorithm applies the matrix and its transpose to structured random vectors to reconstruct the orthonormal row and column bases of judiciously-selected low-rank blocks of the (assumed) butterfly-compressible matrix. The algorithm only requires $O(n^{1.5}\log n)$ computation and $O(n\log n)$ storage resources for matrices arising from the integral equation based discretization of both 2D and 3D Helmholtz problems using either weak or strong admissibility separation criteria. The accuracy of the proposed algorithm only weakly depends on the number of butterfly levels. The computation time of the algorithm can be reduced leveraging distributed-memory parallelism. We expect that the proposed algorithm will play an important role in constructing both dense and sparse, fast and parallel, hierarchical matrix-based direct solvers for high-frequency wave equations. The code described here is part of the Fortran/C++ solver package ButterflyPACK (\url{https://github.com/liuyangzhuan/ButterflyPACK}), freely available online. The integration of butterfly factorizations into the sparse solver package STRUMPACK is currently in progress. 

\section*{Acknowledgments} This research was supported in part by the Exascale Computing Project (17-SC-20-SC), a collaborative effort of the U.S. Department of Energy Office of Science and the National Nuclear Security Administration, and in part by the U.S. Department of Energy, Office of Science, Office of Advanced Scientific Computing Research, Scientific Discovery through Advanced Computing (SciDAC) program through the FASTMath Institute under Contract No. DE-AC02-05CH11231 at Lawrence Berkeley National Laboratory.

This research used resources of the National Energy Research Scientific Computing Center (NERSC), a U.S. Department of Energy Office of Science User Facility operated under Contract No. DE-AC02-05CH11231.

\begin{table*}[!tp]
	\centering	
	\begin{tabular}{|c|c|c|c|c|c|c|}
		\hline	
		$n$ & $L$ &$\epsilon$ & $r$ & error & Time (sec) & Memory (MB)\\		
		\hline
		\hline			
		23806	&8	&1E-02 &85	&1.10E-02 	 &8.03E-01	 &3.78E+01 \\
		\hline
		94024	&10	&1E-02 &122	&1.23E-02 	 &5.69E+00	 &1.56E+02 \\
		\hline
		379065	&12	&1E-02 &173	&1.53E-02 	 &4.78E+01	 &6.60E+02 \\		
		\hline
		1510734	&14	&1E-02 &250	&1.62E-02 	 &3.67E+02	 &2.82E+03 \\	
		\hline
		\hline
		23806	&8	&1E-03 &136	&1.07E-03 	 &1.59E+00	 &7.82E+01 \\	
		\hline	
		94024	&10	&1E-03 &190	&1.25E-03 	 &1.04E+01	 &3.23E+02 \\	
		\hline
		379065	&12	&1E-03 &277	&1.46E-03 	 &8.79E+01	 &1.37E+03 \\	
		\hline
		1510734	&14	&1E-03 &408	&1.59E-03 	 &6.83E+02	 &5.90E+03 \\	
		\hline
		\hline	
		23806	&8	&1E-04 &175	&7.93E-05 	 &2.53E+00	 &1.32E+02 \\	
		\hline	
		94024	&10	&1E-04 &250	&8.85E-05 	 &1.71E+01	 &5.56E+02 \\	
		\hline
		379065	&12	&1E-04 &356	&1.01E-04 	 &1.56E+02	 &2.39E+03 \\	
		\hline
		1510734	&14	&1E-04 &522	&1.12E-04 	 &1.47E+03	 &1.04E+04 \\	
		\hline											
	\end{tabular}
	\caption{Time, memory and measured error for computing a hybrid factorization of \cref{eqn:3dkernel} using the proposed algorithm with varying matrix size $n$ and tolerance $\epsilon$.}\label{tab:EM3D}
\end{table*} 

\begin{figure}[htbp!]
	\begin{center}
		\begin{tabular}{c c}
			\includegraphics[height=1.5in]{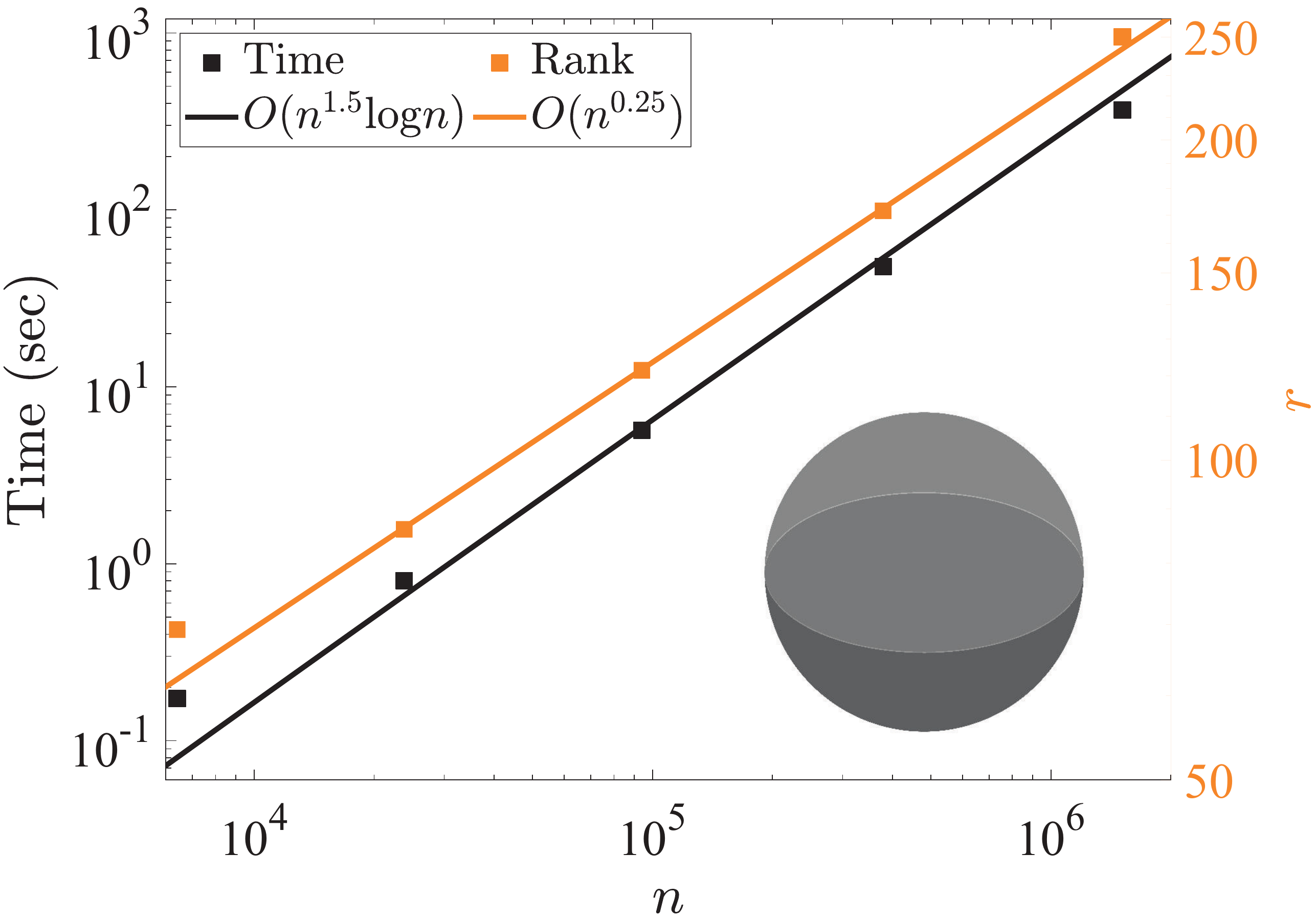} &\includegraphics[height=1.5in]{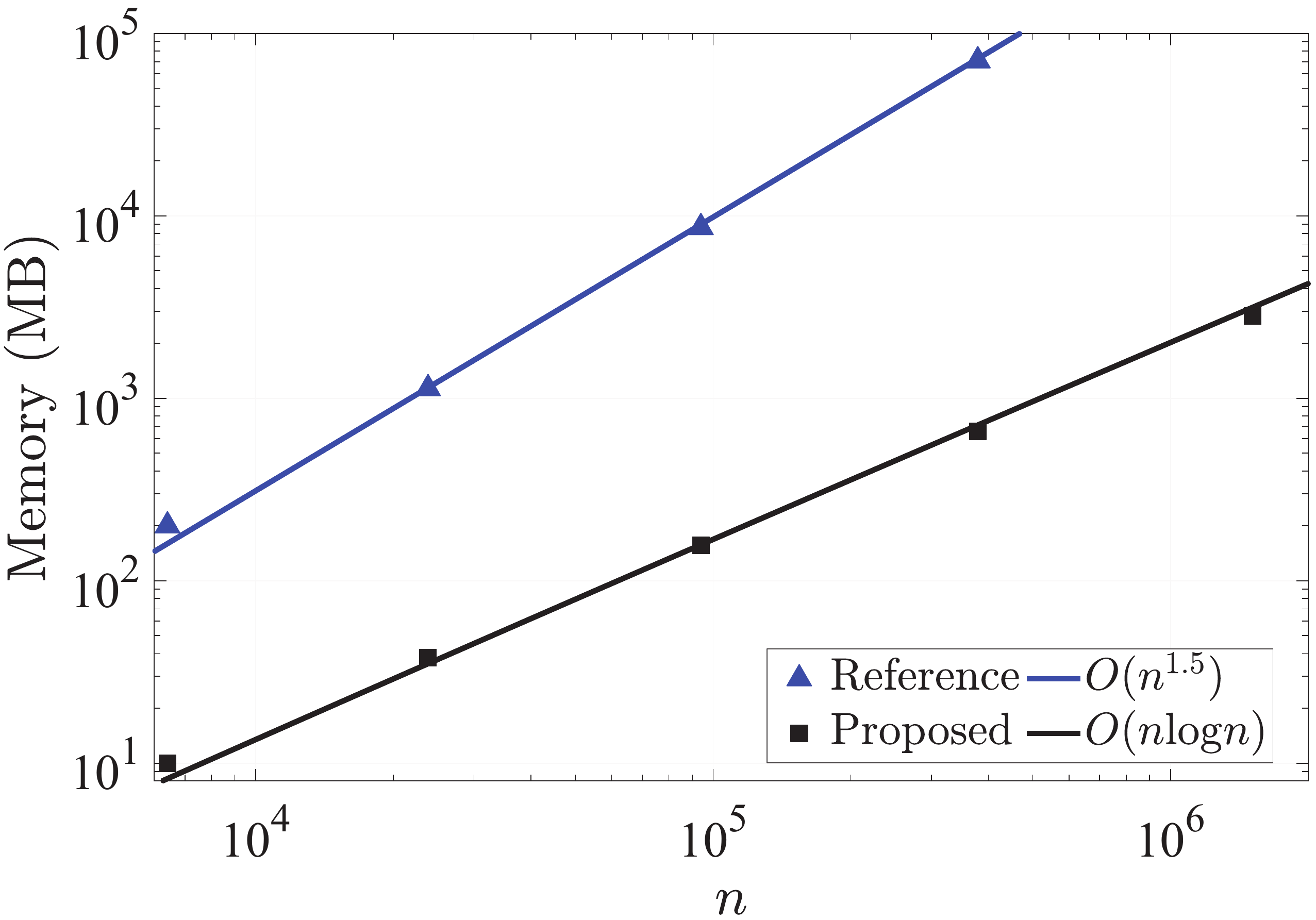}
		\end{tabular}
	\end{center}
	\caption{(a) Computation time (left y-axis), observed rank (right y-axis) and (b) memory for applying \cref{alg:matvec_butterfly} and the reference algorithms to \cref{eqn:3dkernel}. } 
	\label{fig:complexity_EM3D}
\end{figure}

\bibliographystyle{siamplain}
\bibliography{references}

\end{document}